\documentclass[12pt,reqno]{amsart}
\usepackage{amsopn,amssymb,mathrsfs,mathrsfs,a4wide,url}
\usepackage[utf8]{inputenc}

\newcommand{\cl}[1]{\ensuremath{\overline{{#1}}}}
\newcommand{\ds}{\displaystyle}
\newcommand{\ep}{\ensuremath{\varepsilon}}

\newcommand{\lint}[4]{\ensuremath{\int_{#1}^{#2}{#3}\:\mathrm{d}{#4}}}
\newcommand{\lp}[1]{\ensuremath{\ell_{#1}}}

\newcommand{\map}[3]{\ensuremath{{#1}:{#2}\to{#3}}}
\newcommand{\me}{\ensuremath{\mathrm{e}}}
\newcommand{\N}{\mathbb{N}}
\newcommand{\n}[1]{\ensuremath{\left\|{#1}\right\|}}
\newcommand{\ndot}{\ensuremath{\left\|\cdot\right\|}}

\newcommand{\pn}[2]{\ensuremath{\left\|{#1}\right\|_{#2}}}
\newcommand{\pndot}[1]{\ensuremath{\left\|\cdot\right\|}_{#1}}
\newcommand{\R}{\mathbb{R}}
\newcommand{\restrict}[1]{\ensuremath{\!\!\upharpoonright_{#1}}}
\newcommand{\set}[2]{\ensuremath{\left\{{#1}\;:\;\,{#2}\right\}}}
\newcommand{\ts}{\textstyle}
\newcommand{\tn}[1]{\ensuremath{\tri{#1}\tri}}
\newcommand{\tndot}{\ensuremath{\tri\cdot\tri}}
\newcommand{\tri}{{\displaystyle |\kern-.9pt|\kern-.9pt|}}
\newcommand{\ttri}{|\kern-.9pt|\kern-.9pt|}
\newcommand{\ttrin}{\ttri\cdot\ttri}

\DeclareMathOperator{\conv}{conv}
\DeclareMathOperator{\dom}{dom}
\DeclareMathOperator{\ext}{ext}

\DeclareMathOperator{\lspan}{span}

\DeclareMathOperator{\ran}{ran}

\DeclareMathOperator{\supp}{supp}

\numberwithin{equation}{section}

\allowdisplaybreaks

\newtheorem{thm}{Theorem}[section]
\newtheorem{cor}[thm]{Corollary}
\newtheorem{lem}[thm]{Lemma}
\newtheorem{prop}[thm]{Proposition}
\newtheorem{prob}[thm]{Problem}

\theoremstyle{definition}
\newtheorem{defn}[thm]{Definition}
\newtheorem{example}[thm]{Example}

\begin{document}
\title[Approximation of norms on Banach spaces]{Approximation of norms on Banach spaces}
\begin{abstract}
Relatively recently it was proved that if $\Gamma$ is an arbitrary set, then any equivalent norm on $c_0(\Gamma)$ can be approximated uniformly on bounded sets by polyhedral norms and $C^\infty$ smooth norms, with arbitrary precision. We extend this result to more classes of spaces having uncountable symmetric bases, such as preduals of the `discrete' Lorentz spaces $d(w,1,\Gamma)$, and certain symmetric Nakano spaces and Orlicz spaces. We also show that, given an arbitrary ordinal number $\alpha$, there exists a scattered compact space $K$ having Cantor-Bendixson height at least $\alpha$, such that every equivalent norm on $C(K)$ can be approximated as above.
\end{abstract}

\author{Richard J.~Smith}
\address{School of Mathematics and Statistics, University College Dublin, Belfield, Dublin 4, Ireland}
\email{richard.smith@maths.ucd.ie}
\urladdr{http://mathsci.ucd.ie/~rsmith}

\author{Stanimir Troyanski}
\address{Institute of Mathematics and Informatics, Bulgarian Academy of Science, bl.8, acad. G. Bonchev str. 1113 Sofia, Bulgaria, and Departamento de Matem\'aticas, Universidad de Murcia, Campus de Espinardo. 30100 Murcia, Spain}
\email{stroya@um.es}

\thanks{The first author thanks the Institute of Mathematics and Informatics, Bulgarian Academy of Sciences, for its hospitality, when he visited in April 2017. Both authors were supported financially by Science Foundation Ireland under Grant Number `SFI 11/RFP.1/MTH/3112'. The second author was partially supported by MTM2014-54182-P (MINECO/FEDER), MTM2017-86182-P (AEI/FEDER, UE) and Bulgarian National Scientific Fund under Grant DFNI/Russia,01/9/23.06.2017. Both authors are grateful to D.~Leung and G.~Lancien for very helpful remarks.}

\date{\today}
\maketitle

\section{Introduction}

Let $(X,\ndot)$ be a Banach space and let \textbf{P} denote some geometric property of norms, such as strict convexity, polyhedrality or $C^k$-smoothness. We shall say that $\ndot$ can be \emph{approximated} by norms having \textbf{P} if, given $\ep>0$, there exists a norm $\tndot$ on $X$ having \textbf{P}, such that
\[
\n{x} \;\leqslant\; \tn{x} \;\leqslant\; (1+\ep)\n{x},
\]
for all $x \in X$. 

The question of whether all equivalent norms on a given Banach space can be approximated by norms having \textbf{P} has been the subject of a number of papers. A norm on $X$ is called \emph{polyhedral} if, given a finite-dimensional subspace $E \subseteq X$, the restriction of the unit ball to $E$ is a polytope, that is, it has only finitely many extreme points. Given $k \in \N$, a norm is called \emph{$C^k$-smooth} if its $k$th Fr\'echet derivative exists and is continuous at every non-zero point. The norm said to be \emph{$C^{\infty}$-smooth} if this holds for all $k \in \N$. It has been shown that if $X$ is separable and admits a single equivalent polyhedral or $C^k$-smooth norm, where $k \in \N\cup\{\infty\}$, then every equivalent norm on $X$ can be approximated by polyhedral norms or $C^k$-smooth norms, respectively \cite{dfh:98,ht:14}. Other approximation results can be found in \cite{dfh:96,fhz:97,fzz:81,hp:14,pwz:81}.

In the context of approximation by polyhedral and $C^k$-smooth norms, very little was known about the non-separable case until relatively recently. The natural norm of the Banach space $c_0(\Gamma)$, where $\Gamma$ is an arbitrary set, is easily seen to be polyhedral. Moreover, the fact that it admits an equivalent $C^\infty$-smooth norm, at least in the separable case, has been known for some time:~an example of Kuiper was given in \cite{bf:66}. Some results on the approximation of norms on $c_0(\Gamma)$ by $C^\infty$-smooth norms in restricted cases can be found in \cite{fhz:97,pwz:81}. The next result answers the question in the case of $c_0(\Gamma)$ in full generality.

\begin{thm}[{\cite[Theorem 1.7]{bs:16}}]\label{c0_approx} Let $\Gamma$ be an arbitrary set. Then every equivalent norm on $c_0(\Gamma)$ can be approximated by both polyhedral norms and $C^\infty$-smooth norms.
\end{thm}

The proof of this theorem makes use of a number of geometric and topological techniques. It is worth pointing out that this result extends to all subspaces of $c_0(\Gamma)$, because any equivalent norm on such a subspace $X$ can be extended to an equivalent norm on the whole space \cite[Lemma II.8.1]{dgz:93}, and by inspection of the definitions it is clear that the restriction to $X$ of any approximation of the extended norm will inherit the desired properties.

The purpose of this paper is to find new examples of spaces with the property that every equivalent norm on the space can be approximated by both polyhedral norms and $C^\infty$-smooth norms (given the above remark, such examples cannot be isomorphic to subspaces of $c_0(\Gamma)$). In Section \ref{sect_approx_frame} we present the general framework and preparatory lemmas that will be used in Section \ref{sect_main_tools}, in some new approximation theorems are presented. The subsequent sections are devoted to applying these theorems to find new examples of spaces whose norms can be approximated in the manner described above. 

\section{A framework for approximation}\label{sect_approx_frame}

In this section we build the technical tools required in Section \ref{sect_main_tools}. Let $\Gamma$ be an infinite set and let $X$ be a Banach space supporting a system of non-zero projections $(P_\gamma)_{\gamma \in \Gamma}$, satisfying four properties:
\begin{enumerate}
\item $P_\alpha P_\beta = 0$ whenever $\alpha \neq \beta$,
\item $\sup (\n{P_\gamma})_{\gamma \in \Gamma} < \infty$,
\item $X = \cl{\lspan}^{\ndot}(P_\gamma X)_{\gamma \in \Gamma}$, and
\item $X^* = \cl{\lspan}^{\ndot}(P^*_\gamma X^*)_{\gamma \in \Gamma}$.
\end{enumerate}
This concept generalizes the well-known idea of a (shrinking, bounded) \emph{Markushevich basis} (M-basis). If the space $X$ admits a shrinking bounded Markushevich basis (M-basis)  $(e_\gamma,e^*_\gamma)_{\gamma \in \Gamma}$, then defining $P_\gamma x = e^*_\gamma(x)e_\gamma$ yields such a system, where $\dim P_\gamma X = 1$ for all $\gamma \in \Gamma$. Conversely, if each $P_\gamma$ has 1-dimensional range then we obtain such a basis. In view of (4), it makes sense to call any such system of projections shrinking.

To simplify notation, given $x \in X$ and $f \in X^*$, we will denote $P_\gamma x$ and $P^*_\gamma f$ by $x_\gamma$ and $f_\gamma$, respectively. We define the \emph{support} and the \emph{range} of $f \in X^*$ (with respect to the system of projections) to be the sets
\[
\supp(f) \;=\; \set{\gamma \in \Gamma}{f_\gamma \neq 0} \qquad\text{and}\qquad \ran(f) \;=\; \set{\n{f_\gamma}}{\gamma \in \Gamma},
\]
respectively. Likewise, we define the support of $x \in X$ to be
\[
\supp(x) \;=\; \set{\gamma \in \Gamma}{x_\gamma \neq 0}.
\]

\begin{lem}\label{lem_ran(f)_structure}
Given $f \in X^*$ and $\ep>0$, the set $\set{\gamma \in \Gamma}{\n{f_\gamma}\geqslant \ep}$ is finite.
\end{lem}

\begin{proof} According to properties (1), (2) and (4) above, given $f \in X^*$ and $\ep>0$, there exists a finite set $F \subseteq \Gamma$ and $g \in \lspan(P_\gamma^* X^*)_{\gamma \in F} $, such that $\n{f-g} < \ep/D$, where $D:=\sup (\n{f_\gamma})_{\gamma \in \Gamma}$. It follows that $\n{f_\alpha} < \ep$ whenever $\alpha \in \Gamma\setminus F$.
\end{proof}

We proceed to define a series of numerical quantities, sets and `approximating functionals' associated with subsets of $\Gamma$ and elements of $X^*$. First, given finite $F \subseteq \Gamma$, define
\begin{equation}\label{defn_rho}
\rho(F) \;=\; \sup\set{\n{\sum_{\gamma \in F}f_\gamma}}{f \in X^*\text{ and }\n{f_\gamma}\leqslant 1 \text{ whenever }\gamma \in F}.
\end{equation}
Evidently, $\rho$ is increasing, in the sense that $\rho(F)\leqslant \rho(G)$ whenever $F \subseteq G \subseteq \Gamma$ and $G$ is finite.

\begin{lem}\label{lem_lim_inf} Let $(F^\alpha)$ be a net of finite subsets of $\Gamma$, such that
\[
A \;:=\; \lim \inf F_\alpha \;=\; \bigcup_\alpha \bigcap_{\beta \geqslant \alpha} F_\beta,
\]
is infinite. Then $\rho(F^\alpha) \to \infty$.
\end{lem}

\begin{proof} Suppose that $\rho(F^\alpha) \not\to \infty$. By taking a subnet if necessary, we can assume that there exists some $L>0$ such that $\rho(F^\alpha) \leqslant L$ for all $\alpha$ (note that $A$ remains infinite). Since the $P_\gamma$ are non-zero, for each $\gamma \in \bigcup_\alpha F^\alpha$, we can select $g_\gamma \in P_\gamma^* X^*$ such that $\n{g_\gamma}=1$. Given $\alpha$, define $f^\alpha = \sum_{\gamma \in F^\alpha} g_\gamma$. By property (1) of our system of projections, we see that $f^\alpha_\gamma = P_\gamma^* f^\alpha= g_\gamma$ if $\gamma \in F^\alpha$, and $f^\alpha_\gamma=0$ otherwise. We have $\n{f^\alpha} \leqslant \rho(F^\alpha) \leqslant L$ for all $n$, and thus the net $(f^\alpha)$ admits a $w^*$-accumulation point $f$. Given the definition of $A$ and the $w^*$-$w^*$-continuity of each $P_\gamma^*$, we conclude that $f_\gamma=g_\gamma$ for all $\gamma \in A$. However, the fact that $\n{f_\gamma}=1$ for infinitely many $\gamma \in \Gamma$ violates Lemma \ref{lem_ran(f)_structure}.
\end{proof}

Next, we define a series of numerical quantities, subsets and approximating functionals associated with elements of $X^*$. Let $f \in X^*$ and $k \in \N$. It follows from Lemma \ref{lem_ran(f)_structure} that if $\ran(f)$ is infinite then it has a single accumulation point at $0$. Therefore, if $|\ran(f)| \geqslant k$, it is legitimate to define $p_k(f)$ to be the $k$th largest value of $\ran(f)$. If $|\ran(f)| < k$, set $p_k(f)=0$. It is clear that $p_k(f) \to 0$ as $k \to \infty$. 

We continue by defining
\[
q_k(f) \;=\; p_k(f)-p_{k+1}(f), \qquad G_k(f) \;=\; \set{\gamma \in \supp(f)}{\n{f_\gamma} \geqslant p_k(f)},
\]
and
\[
H_k(f) \;=\; \set{\gamma \in \supp(f)}{\n{f_\gamma} = p_k(f)} \;=\; G_k(f)\setminus G_{k-1}(f),
\]
where we set $G_0(f)=\varnothing$ for convenience. Evidently, $q_k(f)>0$ if and only if $p_k(f)>0$. Also, $G_k(f)$ is always finite. If $p_k(f)>0$ then $G_k(f)$ is finite by Lemma \ref{lem_ran(f)_structure}. If $p_k(f)=0$ then $G_k(f)=\supp(f)$, which must be finite in this case. It is obvious that the $G_k(f)$ form an increasing sequence of subsets of $\Gamma$.

Now define a function $\map{\theta}{X^*}{[0,\infty]}$ by 
\[
\theta(f) \;=\; \sum_{k=1}^\infty q_k(f)\rho(G_k(f)).
\]
Given $t \neq 0$, it is clear that $q_k(t f)=|t|q_k(f)$ and $G_k(t f)=G_k(f)$, and thus $\theta$ is absolutely homogeneous.

\begin{lem}\label{lem_w^*-lsc} The function $\theta$ is $w^*$-lower semicontinuous.
\end{lem}

\begin{proof} Let $f \in X^*$, $\lambda \geqslant 0$ and assume $\theta(f) > \lambda$.  Fix minimal $n\in\N$ such that
\[
\alpha:=\sum_{k=1}^n q_k(f)\rho(G_k(f)) > \lambda.
\]
By the minimality of $n$, $p_n(f) \geqslant q_n(f)>0$. Set
\[
\ep \;=\; \min\left\{\frac{\alpha-\lambda}{\rho(G_n(f))},\, p_n(f)\right\}>0.
\]
By the $w^*$-$w^*$-continuity of each $P_\gamma^*$ and the natural lower semicontinuity of dual norms, the functions $g \mapsto \n{g_\gamma}$, $\gamma \in \Gamma$, are $w^*$-lower semicontinuous. Therefore, as $G_n(f)$ is finite, the set
\[
U:=\set{g \in X^*}{\n{g_\gamma} \;>\; \n{f_\gamma} - \ep \text{ whenever }\gamma \in G_n(f)},
\]
is $w^*$-open (and clearly contains $f$). 

Let $g \in U$. Given $1 \leqslant k \leqslant n$, let $\beta_k = \min\set{\n{g_\gamma}}{\gamma \in G_k(f)}$. From the definition of $\ep$ and $U$, we know that
\[
\beta_k \;>\; p_k(f) - \ep \;\geqslant\; 0,
\]
for all such $k$. Clearly the $\beta_k$ are non-increasing. Now find $j_k \in \N$, $1 \leqslant k \leqslant n$, such that $p_{j_k}(g) = \beta_k$. Observe that $G_k(f)\subseteq G_{j_k}(g)$ whenever $1 \leqslant k \leqslant n$. Since the $\beta_k$ are non-increasing, the integers $j_k$ must be non-decreasing. Because $\rho$ is an increasing function, as observed above, it follows that
\begin{align*}
\theta(g) &\;>\; \sum_{k=1}^{n-1}\sum_{i=j_k}^{j_{k+1}-1} q_i(g)\rho(G_i(g)) + \sum_{i=j_n}^\infty q_i(g)\rho(G_i(g)) \\
&\geqslant\; \sum_{k=1}^{n-1} \rho(G_{j_k}(g))\sum_{i=j_k}^{j_{k+1}-1}(p_i(g)-p_{i+1}(g)) + \rho(G_{j_n}(g))\sum_{i=j_n}^\infty (p_i(g)-p_{i+1}(g))\\
&=\;  \sum_{k=1}^{n-1} \rho(G_{j_k}(g))(\beta_k-\beta_{k+1}) + \rho(G_{j_n}(g))\beta_n\\
&\geqslant\;  \sum_{k=1}^{n-1} \rho(G_k(f))(\beta_k-\beta_{k+1}) + \rho(G_n(f))\beta_n\\
&=\; \beta_1 \rho(G_1(f)) + \sum_{k=2}^n \beta_k(\rho(G_k(f))-\rho(G_{k-1}(f)))\\
&>\; (p_1(f)-\ep)\rho(G_1(f)) + \sum_{k=2}^n (p_k(f)-\ep)(\rho(G_k(f))-\rho(G_{k-1}(f)))\\
&=\; \sum_{k=1}^{n-1} (p_k(f)-p_{k+1}(f))\rho(G_k(f)) + p_n(f)\rho(G_n(f)) - \ep\rho(G_n(f))\\
&\geqslant\; \sum_{k=1}^n q_k(f)\rho(G_k(f)) + \lambda - \alpha \;=\; \lambda. \qedhere
\end{align*}
\end{proof}

Given $f \in X^*$ and $m,n \in \N$, $m<n$, we define associated functionals $h_m(f)$, $g_{m,n}(f)$ and $j_{m,n}(f)$. The $j_{m,n}(f)$ will be the functionals that we use to approximate $f$ in Section \ref{sect_main_tools}. First, given $k \in \N$, define the auxiliary functionals
\[
\omega_k(f) \;=\; \sum_{\gamma \in G_k(f)}  \frac{f_\gamma}{\n{f_\gamma}}.
\]
By the definition of $\rho$, $\n{\omega_k(f)}\leqslant\rho(G_k(f))$ for all $k \in \N$. Now define
\[
h_m(f) \;=\; \sum_{k=1}^m q_k(f)\omega_k(f).
\]
The following straightforward observation will be used in the next lemma:
\begin{equation}\label{eqn_h_less_than_theta}
\n{h_m(f)} \;\leqslant\; \sum_{k=1}^m q_k(f)\n{\omega_k(f)} \;\leqslant\; \sum_{k=1}^m q_k(f)\rho(G_k(f)) \;\leqslant\; \theta(f),
\end{equation}
for all $m$.

\begin{lem}\label{lem_theta_dominate} We have  $\n{f} \leqslant \theta(f)$ for all $f \in X^*$, and $\n{f - h_m(f)} \to 0$ whenever $\theta(f) < \infty$.
\end{lem}

\begin{proof} Assume that $f \in X^*$ is non-zero and that $\theta(f)<\infty$. Given $\gamma \in G_m(f)$, let $k_\gamma$ denote the unique index $k  \leqslant m$ for which $\gamma \in H_k(f)$. Assuming $p_m(f)>0$, we have
\begin{align}\label{eqn_h}
h_m(f) &\;=\; \sum_{k=1}^m q_k(f)\left(\sum_{\gamma \in G_k(f)} \frac{f_\gamma}{\n{f_\gamma}} \right)\nonumber \\
&\;=\; \sum_{\gamma \in G_m(f)} \frac{f_\gamma}{\n{f_\gamma}} \left(\sum_{k = k_\gamma}^m q_k(f) \right)\nonumber \\
&\;=\; \sum_{\gamma \in G_m(f)} \frac{f_\gamma}{\n{f_\gamma}}(p_{k_\gamma}(f)-p_{m+1}(f)) \;=\; \left(\sum_{\gamma \in G_m(f)} f_\gamma\right) - p_{m+1}(f)\omega_m(f).
\end{align}
If $\ran(f)$ is finite, then the definition of $h_m(f)$, together with (\ref{eqn_h}), shows that $h_m(f)=f$ whenever $m \geqslant |\ran(f)|-1$. Coupling this with (\ref{eqn_h_less_than_theta}) yields the conclusion.

For the remainder of the proof, we assume that $\ran(f)$ is infinite and thus that (\ref{eqn_h}) applies for all $m \in \N$. According to properties (1) and (3) of our system of projections above, given $\ep>0$, there exists a finite set $F\subseteq \Gamma$ and $x \in X$ such that $\n{x}=1$, $x \;=\; \sum_{\gamma \in F} x_\gamma$ and $\n{f} \leqslant f(x) + \ep$. Choose $M\in \N$ large enough so that $F \cap G_m(f) = F \cap G_M(f)$ whenever $m \geqslant M$. Observe that if $\gamma \in F\setminus G_M(f)$, then $f_\gamma=0$. If not, then as $p_m(f)\to 0$, there must exist $m\geqslant M$ such that $\gamma \in G_m(f)$, giving $\gamma \in F \cap G_m(f)=F \cap G_M(f)$, which is false.  Hence, using the above facts, given $m \geqslant M$, we have 
\begin{align*}
\n{f}-\ep \;\leqslant\; f(x) &\;=\; \sum_{\gamma \in F} f_\gamma(x)\\
&\;=\; \sum_{\gamma \in F \cap G_m(f)} f_\gamma(x) + \sum_{F\setminus G_m(f)} f_\gamma(x) + \sum_{G_m(f)\setminus F} f_\gamma(x) \\
&\;=\; \sum_{\gamma \in G_m(f)} f_\gamma(x) \\
&\;=\; h_m(f)(x) + p_{m+1}(f)\omega_n(f)(x)\\
&\;\leqslant\; \theta(f) + p_{m+1}(f)\omega_M(f)(x) \;\to\; \theta(f),
\end{align*}
using equation (\ref{eqn_h_less_than_theta}) above, the fact that $\omega_m(f)(x)=\omega_M(f)(x)$ (by the choice of $M$), and as $p_m(f) \to 0$ as $m\to\infty$. Since this holds for all $\ep > 0$, it follows that $\n{f} \leqslant \theta(f)$.

Now assume that $\theta(f) < \infty$. Fix $m\in\N$ and set $g=f-h_m(f)$. We observe that $p_k(g)=p_{m+k}(f)$ and $G_k(g)=G_{m+k}(f)$ for all $k \in \N$. Therefore, using the above applied to $g$ yields
\begin{equation}\label{eqn_f_h_m(f)}
\n{f-h_m(f)} \;\leqslant\; \theta(g) \;=\; \sum_{k=1}^\infty q_k(g)\rho(G_k(g)) \;=\; \sum_{k=m+1}^\infty q_k(f)\rho(G_k(f)).
\end{equation}
Because $\theta(f)<\infty$, the quantity on the right hand side tends to $0$ as $m\to\infty$.
\end{proof}

We continue by defining, for $m,n \in \N$, $m<n$,
\[
g_{m,n}(f) \;=\; \begin{cases} {\ds \frac{\theta(f-h_m(f))}{\rho(G_n(f))}}\omega_n(f)  & \text{if $p_n(f) > 0$}\\ 0 & \text{otherwise,}\end{cases} 
\]
and the approximating functionals
\[
j_{m,n}(f) \;=\; h_m(f) + g_{m,n}(f).
\]
One of the principal reasons for considering the $j_{m,n}(f)$ is brought to light in the next lemma.

\begin{lem}\label{lem_convex_combination} Let $f \in X^*$, $m \in \N$, and suppose that $\theta(f)<\infty$. Then there exist constants $\lambda_n \geqslant 0$, $n>m$, such that
\[
\sum_{n=m+1}^\infty \lambda_n \;=\; 1 \qquad\text{and}\qquad f \;=\; \sum_{n=m+1}^\infty \lambda_n j_{m,n}(f).
\]
\end{lem}

\begin{proof}
Fix $m \in \N$ and assume $\theta(f)<\infty$. We know, from Lemma \ref{lem_theta_dominate} and (\ref{eqn_f_h_m(f)}), that
\begin{equation}
f \;=\; \sum_{n=1}^\infty q_n(f)\omega_n(f) \qquad\text{and}\qquad \theta(f-h_m(f)) \;=\; \sum_{n=m+1}^\infty q_n(f)\rho(G_n(f)). \label{eqn_infinite_sums}
\end{equation}
There are two cases. If $f=h_m(f)$, then $p_n(f)=0$ whenever $n>m$, which implies that $j_{m,n}(f)=h_m(f)=f$ for all such $n$. In this case, it is clearly sufficient to let $\lambda_{m+1}=1$ and $\lambda_n=0$ for $n>m+1$. Instead, if $f\neq h_m(f)$ then $\theta(f-h_m(f)) > 0$. This time, given $n>m$, set
\[
\lambda_n \;=\; \frac{q_n(f)\rho(G_n(f))}{\theta(f-h_m(f))}.
\]
By (\ref{eqn_infinite_sums}) and Lemma \ref{lem_theta_dominate}, we have $\sum_{n=m+1}^\infty \lambda_n = 1$ and
\begin{align*}
\sum_{n=m+1}^\infty \lambda_n j_{m,n}(f) \;&=\; \sum_{n=m+1}^\infty \lambda_n h_m(f) + q_n(f)\omega_n(f)\\
&=\; h_m(f) + \sum_{n=m+1}^\infty q_n(f)\omega_n(f) \;=\; f. \tag*{\qedhere}
\end{align*}
\end{proof}

In our final lemma, we consider what happens when we have a net $j_{m^\alpha,n^\alpha}(f^\alpha)$ of the approximating functionals converging to an infinitely supported element.

\begin{lem}\label{lem_net_1} Let $d \in X^*$ have infinite support, and consider a net $j^\alpha := j_{m^\alpha,n^\alpha}(f^\alpha)$, such that $j^\alpha \stackrel{w^*}{\to} d$. Suppose moreover that $\sup_\alpha \theta(f^\alpha) <\infty$. Then $f^\alpha \stackrel{w^*}{\to} d$.
\end{lem}

\begin{proof} Let $L > 0$ such that $\theta(f^\alpha) \leqslant L$ for all $\alpha$. The proof will be comprised of a number of steps. In the first step, we show that $\rho(G_{n^\alpha}(f^\alpha)) \to \infty$. Let $\gamma \in \supp(d)$. As $P^*_\gamma$ is $w^*$-$w^*$-continuous, $j^\alpha_\gamma \stackrel{w^*}{\to} d_\gamma \neq 0$. In particular, there exists $\alpha$ such that $j^\beta_\gamma \neq 0$ for all $\beta \geqslant \alpha$. It follows that
\[
\supp(d) \;\subseteq\; \bigcup_\alpha \bigcap_{\beta \geqslant \alpha} \supp(j^\beta) \;\subseteq\; \bigcup_\alpha \bigcap_{\beta \geqslant \alpha} G_{n^\beta}(f^\beta).
\]
Consequently, $\rho(G_{n^\alpha}(f^\alpha)) \to \infty$ by Lemma \ref{lem_lim_inf}.

In the second step, we use the first step to prove that $g_{m^\alpha,n^\alpha}(f^\alpha) \stackrel{w^*}{\to} 0$. Fix $\gamma \in \Gamma$. Since $\n{\omega_n(f)_\gamma} \leqslant 1$, we observe that if $p_{n^\alpha}(f^\alpha)>0$ then
\[
\n{g_{m^\alpha,n^\alpha}(f^\alpha)_\gamma} \;\leqslant\; \frac{\theta(f^\alpha-h_{m^\alpha}(f^\alpha))}{\rho(G_{n^\alpha}(f^\alpha))} \;\leqslant\; \frac{\theta(f^\alpha)}{\rho(G_{n^\alpha}(f^\alpha))} \;\leqslant\; \frac{L}{\rho(G_{n^\alpha}(f^\alpha))},
\]
and $\n{g_{m^\alpha,n^\alpha}(f^\alpha)_\gamma}=0$ otherwise. In any case, it follows that
\begin{equation}\label{eqn_g_1}
\n{g_{m^\alpha,n^\alpha}(f^\alpha)_\gamma} \to 0.
\end{equation}
 Moreover, by the definition of $\rho$,
\begin{equation}\label{eqn_g_2}
\n{g_{m^\alpha,n^\alpha}(f^\alpha)} \;=\; \frac{\theta(f^\alpha-h_{m^\alpha}(f^\alpha))\n{\omega_{n^\alpha}(f^\alpha)}}{\rho(G_{n^\alpha}(f^\alpha))} \;\leqslant\; \theta(f^\alpha-h_{m^\alpha}(f^\alpha)) \;\leqslant\; L,
\end{equation}
if $p_{n^\alpha}(f^\alpha)>0$, and if not then of course this inequality still holds. Hence, by considering (\ref{eqn_g_1}), (\ref{eqn_g_2}) and property (3) of our system of projections, we deduce that $g_{m^\alpha,n^\alpha}(f^\alpha) \stackrel{w^*}{\to} 0$.

Of course, this implies that $h_{m^\alpha}(f^\alpha) \stackrel{w^*}{\to} d$. We complete the proof by showing, in this final step, that
\begin{equation}\label{eqn_h_limit}
f^\alpha - h_{m^\alpha}(f^\alpha) \;\stackrel{w^*}{\to}\; 0.
\end{equation}
We repeat the first step, substituting $G_{m^\alpha}(f^\alpha)$ for $G_{n^\alpha}(f^\alpha)$ and noting that $\supp (h_{m^\alpha}(f^\alpha))$ $\subseteq G_{m^\alpha}(f^\alpha)$, to obtain $\rho(G_{m^\alpha}(f^\alpha)) \to \infty$. Since $\rho$ is increasing, given arbitrary $f\in X^*$ and $m\in\N$ we have
\begin{equation}\label{eqn_f_1}
p_m(f)\rho(G_m(f)) \;=\; \sum_{n=m}^\infty q_n(f)\rho(G_m(f)) \;\leqslant\; \sum_{n=m}^\infty q_n(f)\rho(G_n(f)) \;\leqslant\; \theta(f).
\end{equation}
If we fix $\gamma \in \Gamma$, then by (\ref{eqn_h}) and (\ref{eqn_f_1}), whenever $p_{m^\alpha}(f^\alpha)>0$ we have
\[
\n{f^\alpha_\gamma - h_{m^\alpha}(f^\alpha)_\gamma} \;=\; \n{p_{m^\alpha+1}(f^\alpha)\omega_{m^\alpha}(f^\alpha)_\gamma} \;\leqslant\; p_{m^\alpha+1}(f^\alpha) \;\leqslant\; \frac{L}{\rho(G_{m^\alpha}(f^\alpha))},
\]
and $f^\alpha = h_{m^\alpha}(f^\alpha)$ otherwise. Therefore
\begin{equation}\label{eqn_f_2}
\n{f^\alpha_\gamma - h_{m^\alpha}(f^\alpha)_\gamma} \;\to\; 0,
\end{equation}
for all $\gamma \in \Gamma$. Moreover,
\begin{equation}\label{eqn_f_3}
\n{f^\alpha - h_{m^\alpha}(f^\alpha)} \;\leqslant\; \theta(f^\alpha-h_{m^\alpha}(f^\alpha)) \;\leqslant\; \theta(f^\alpha) \;\leqslant\; L.
\end{equation}
Therefore, by considering (\ref{eqn_f_2}), (\ref{eqn_f_3}) and again property (3) of our projections, we obtain (\ref{eqn_h_limit}) as required.
\end{proof}

\section{Approximation of norms}\label{sect_main_tools}

We begin this section with a few more preliminaries. Given a Banach space $(X,\ndot)$ and a bounded set $C \subseteq X^*$, recall that $B \subseteq C$ is called a \emph{James boundary} of $C$ if, given $x \in X$, there exists $b \in B$ such that
\[
b(x) \;=\; \sup\set{f(x)}{f \in C}.
\]
Hereafter, we will simply refer to a James boundary of a given set as a \emph{boundary}. Given an equivalent norm $\tndot$ on $X$, we say that $B$ is a boundary of $(X,\tndot)$, or simply $\tndot$, if it is a boundary of $B_{(X,\ttrin)^*}$, i.e.~$\tn{b}\leqslant 1$ for all $b \in B$ and, given $x \in X$, there exists $b \in B$ such that $b(x)=\tn{x}$. For example, by the Hahn-Banach and Krein-Milman Theorems, the set of extreme points $\ext(B_{(X,\ttrin)^*})$ is always a boundary of $\tndot$.

As well as boundaries, we recall the notion of norming sets. Let $V$ be a subspace of $X$. A bounded set $E\subseteq X^*$ is called \emph{norming for $V$} if there exists $r>0$ such that
\[
\sup\set{f(x)}{f \in E} \;\geqslant\; r\n{x},
\]
for all $x \in V$. When $V=X$ we just say that $E$ is norming.

We will require the following notion, which has been used to provide a sufficient condition for the existence of equivalent polyhedral and $C^\infty$-smooth norms.

\begin{defn}[{\cite[Definition 11]{fpst:14}}]\label{defn_lrc}
Let $X$ be a Banach space. We say that $E \subseteq X^*$ is \emph{$w^*$-locally relatively norm-compact} ($w^*$-LRC for short) if, given $f \in E$, there exists a $w^*$-open set $U \subseteq X^*$, such that $f \in U$ and $E \cap U$ is relatively norm-compact. We say that $E$ is \emph{$\sigma$-$w^*$-LRC} if it can be expressed as the union of countably many $w^*$-LRC sets.
\end{defn}

The following result motivates their introduction. We will call a subset $E \subseteq X^*$ $w^*$-$K_\sigma$ if it is the union of countably many $w^*$-compact sets.

\begin{thm}[{\cite[Theorem 2.1]{b:14}} and {\cite[Theorem 7]{fpst:14}}]\label{thm_tb}
Let $(X,\ndot)$ be a Banach space having a boundary that is both $\sigma$-$w^*$-LRC and $w^*$-$K_\sigma$. Then $\ndot$ can be approximated by both $C^\infty$-smooth norms and polyhedral norms.
\end{thm}

If $(e_\gamma,e^*_\gamma)_{\gamma \in \Gamma}$ is a bounded M-basis, then $E:=\set{e^*_\gamma}{\gamma \in \Gamma} \cup \{0\}$ is both $\sigma$-$w^*$-LRC and $w^*$-$K_\sigma$, as is $\lspan(E)$. More information concerning the topological behaviour of $w^*$-LRC sets can be found in \cite{smith:17}. We quote one result from that study here.

\begin{thm}[{\cite[Theorem 2.3]{smith:17}}]\label{thm_linspan}
If $E$ is a $\sigma$-$w^*$-LRC and $w^*$-$K_\sigma$ subset of a dual Banach space $X^*$, then so is the subspace $\lspan(E)$.
\end{thm}

Now let the Banach space $X$ support a system of projections $(P_\gamma)_{\gamma \in \Gamma}$ as in Section \ref{sect_approx_frame}. Suppose that, for each $\gamma \in \Gamma$, we are given a linear (but not necessarily closed) subspace $A_\gamma \subseteq P^*_\gamma X^*$. Given finite $F \subseteq \Gamma$, we define additional subspaces
\begin{align*}
V_F \;&=\; \lspan(P_\gamma X)_{\gamma \in F} \;=\; \set{x \in X}{\supp(x) \subseteq F},\\
W_F \;&=\; \lspan(P^*_\gamma X^*)_{\gamma \in F}  \;=\; \set{f \in X^*}{\supp(f)\subseteq F}, \text{ and}\\
A_F \;&=\; \lspan(A_\gamma)_{\gamma \in F}.
\end{align*}
Bearing in mind property (1) of our system of projections $(P_\gamma)_{\gamma \in \Gamma}$, by standard methods we see that $B_{W_F}$ is norming for $V_F$.

\begin{defn}\label{defn_admissible_set} We will call the family $(A_\gamma)_{\gamma \in \Gamma}$ \emph{admissible} if, given $\ep>0$, a finite set $F\subseteq \Gamma$, and a $w^*$-compact subset $C \subseteq W_F$ that is norming for $V_F$, there exists a set $D \subseteq W_F$ such that
\[
C \;\subseteq\; \cl{D}^{w^*} \;\subseteq\; C + \ep B_{W_F} \qquad\text{and}\qquad \cl{D}^{w^*} \cap A_F \text{ is a boundary of }\cl{D}^{w^*}.
\]
\end{defn}

It is obvious that the family $(P_\gamma^* X^*)_{\gamma \in \Gamma}$ is admissible. We will make use of this fact, together with other examples, later on.

Now we can state and prove our main tool.

\begin{thm}\label{thm_main_tool} Let $(A_\gamma)_{\gamma \in \Gamma}$ be an admissible family as above, and let $B \subseteq C \subseteq X^*$ be sets such that $C$ is $w^*$-compact and norming, and $B$ is a boundary of $C$ with the property that $\theta(f) < \infty$ whenever $f \in B$. Then, given $\ep>0$, there exists $D \subseteq X^*$ such that
\[
B \;\subseteq\; \cl{D}^{w^*} \;\subseteq\; C + \ep B_{X^*} \qquad\text{and}\qquad \cl{D}^{w^*} \cap A \text{ is a boundary of }\cl{D}^{w^*},
\]
where $A:=\lspan(A_\gamma)_{\gamma \in \Gamma}$. 
\end{thm}

\begin{proof} 
By rescaling $C$ if necessary, we can assume that $C \subseteq B_{X^*}$. Since $C$ is norming and a subset of $B_{X^*}$, there exists $r \in (0,1]$ such that $\sup\set{f(x)}{f \in C} \geqslant r\n{x}$ for all $x \in X$. Let $\ep \in (0,1)$. Given $k \in \N$, define
\[
C_k \;=\; \set{f \in C}{\theta(f) \leqslant k}.
\]
By Lemma \ref{lem_w^*-lsc}, this set is $w^*$-compact. Since $\theta(f)$ is finite for all $f \in B$, we know that $B\subseteq \bigcup_{k=1}^\infty C_k$. Define sets of approximating functionals
\begin{equation}\label{defn_J_k}
J_k \;=\; \set{j_{m,n}(f)}{f \in C_k,\; m,n \in \N,\;m<n\text{ and } \theta(f-h_m(f))<2^{-k-2}r\ep}.
\end{equation}
Following Lemma \ref{lem_theta_dominate} and (\ref{defn_rho}),
\begin{equation}\label{est_j_mn}
\n{f-j_{m,n}(f)} \;\leqslant\; \n{f-h_m(f)} + \n{g_{m,n}(f)} \;\leqslant\; 2\theta(f-h_m(f)) \;\to\; 0,
\end{equation}
as $m \to \infty$. Thus we see that
\begin{equation}\label{est_J_k}
C_k \;\subseteq\; \cl{J_k}^{\ndot} \;\subseteq\; \cl{J_k}^{w^*} \;\subseteq\; C_k + 2^{-k-2}r\ep B_{X^*}.
\end{equation}
Given non-empty finite $F \subseteq \Gamma$, define the $w^*$-compact set
\begin{equation}\label{defn_C_kF}
C_{k,F} \;=\; \set{j \in \cl{J_k}^{w^*}}{\supp(j) \subseteq F} \;\subseteq \; W_F.
\end{equation}
We apply Definition \ref{defn_admissible_set} to $\frac{1}{2}\cdot 3^{-k-|F|}\ep$ and $C_{k,F} + \frac{1}{2}\cdot 3^{-k-|F|}\ep B_{W_F}$ (which is norming for $V_F$) to obtain a set $D_{k,F} \subseteq W_F$ such that
\begin{equation}\label{est_D_kF}
{\ts C_{k,F} + \frac{1}{2}\cdot 3^{-k-|F|}\ep B_{W_F}} \;\subseteq\; \cl{D_{k,F}}^{w^*} \;\subseteq\; C_{k,F} + 3^{-k-|F|}\ep B_{W_F},
\end{equation}
and
\begin{equation}\label{D_kF_boundary}
\cl{D_{k,F}}^{w^*} \cap A_F \text{ is a boundary of }\cl{D_{k,F}}^{w^*}.
\end{equation}
Now define
\[
D \;=\; \bigcup\set{(1+ 2^{-k}\ep)D_{k,F}}{F \subseteq \Gamma,\;k, |F| \in \N}.
\]

First, we show that $B \subseteq \cl{D}^{w^*} \subseteq C + \ep B_{X^*}$. By (\ref{est_J_k}), (\ref{defn_C_kF}), (\ref{est_D_kF}) and the fact that $C \subseteq B_{X^*}$, 
\[
D_{k,F} \;\subseteq\; C_{k,F} + 3^{-k-|F|}\ep B_{W_F} \;\subseteq\; \cl{J_k}^{w^*} + {\ts\frac{1}{9}}\ep B_{X^*} \;\subseteq\; C + {\ts\frac{1}{4}}\ep B_{X^*} \;\subseteq\; {\ts\frac{5}{4}}B_{X^*},
\]
giving
\[
(1 + 2^{-k}\ep)D_{k,F} \;\subseteq\; C + {\ts\frac{1}{4}}\ep B_{X^*} + {\ts\frac{5}{8}}\ep B_{X^*} \;\subseteq\; C+\ep B_{X^*}.
\]
As $C+\ep B_{X^*}$ is $w^*$-closed, it follows that $\cl{D}^{w^*} \subseteq C + \ep B_{X^*}$. Now let $f \in B$ and fix $k \in \N$ large enough so that $f \in C_k$. Given $\ell \geqslant k$, it follows that
\begin{align*}
f \in C_\ell \;\subseteq\; \cl{J_\ell}^{w^*} \;&\subseteq\; \cl{\bigcup\set{C_{\ell,|F|}}{F\subseteq \Gamma,\;|F| \in \N}}^{w^*}\\
&\subseteq\; \cl{\bigcup\set{D_{\ell,|F|}}{F\subseteq \Gamma,\;|F| \in \N}}^{w^*},
\end{align*}
and thus $(1+2^{-\ell}\ep)f \in \cl{D}^{w^*}$. Consequently, $f \in \cl{D}^{w^*}$, and we have $B \subseteq \cl{D}^{w^*}$ as required.

Now we show that $\cl{D}^{w^*} \cap A$ is a boundary of $\cl{D}^{w^*}$. Fix $x \in X$, $\n{x}=1$, and set
\[
\eta \;=\; \sup\set{f(x)}{f \in \cl{D}^{w^*}}.
\]
Evidently, as $B$ is a boundary of $C$ and $B \subseteq  \cl{D}^{w^*}$, we have $\eta \geqslant r > 0$.

We need to find an element $e \in \cl{D}^{w^*} \cap A$ satisfying $e(x)=\eta$. By $w^*$-compactness, there exists $d \in \cl{D}^{w^*}$ such that $d(x)=\eta$. Our first task is to show that there exists $k \in \N$ having the property that
\begin{equation}\label{d_fixed_k}
d \in \cl{\bigcup\set{(1+ 2^{-k}\ep)D_{k,F}}{F\subseteq \Gamma,\;|F| \in \N}}^{w^*}.
\end{equation}
Observe that if there is no $k \in \N$ for which (\ref{d_fixed_k}) holds, then 
\begin{equation}\label{est_d_intersection}
d \in \bigcap_{k=1}^\infty \left(\cl{\bigcup\set{(1+ 2^{-\ell}\ep)D_{k,F}}{F\subseteq \Gamma,\; \ell, |F| \in \N,\; \ell\geqslant k}}^{w^*}\right).
\end{equation}
We show that the hypotheses preclude this possibility from taking place. We start by setting $\xi = \sup\set{f(x)}{f \in C} \geqslant r >0$. Given $k \in \N$, (\ref{est_D_kF}) and (\ref{est_d_intersection}) imply
\[
d(x) \;\leqslant\; (1+ 2^{-k}\ep)(\xi + 3^{-k-1}\ep).
\]
As this holds for all $k\in\N$, we have $d(x)\leqslant \xi$. On the other hand, as $B$ is a boundary of $C$, there exists $k \in \N$ and $b \in C_k$, such that $b(x)=\xi$. Using the definition of $J_k$ and (\ref{est_j_mn}), there exists $m \in \N$ such that $j_{m,m+1}(b) \in J_k$ and
\[
j_{m,m+1}(b)(x) \;\geqslant\; b(x) - \n{b-j_{m,n}(b)} \;>\; \xi - 2^{-k-2}r\ep. 
\]
Since $(1+2^{-k}\ep)J_k \subseteq \cl{D}^{w^*}$, and recalling that $\ep<1$, we have
\begin{align*}
\eta \;\geqslant\; (1+2^{-k}\ep)j_{m,m+1}(b)(x) \;&>\; (1+2^{-k}\ep)(\xi - 2^{-k-2}r\ep)\\
&\geqslant\; (1+2^{-k}\ep)(1-2^{-k-2}\ep)\xi\\
&=\; (1+{\ts\frac{3}{4}}\cdot 2^{-k}\ep - 2^{-2k-2}\ep^2)\xi\\
&>\; (1+{\ts\frac{5}{8}}\cdot 2^{-k}\ep)\xi \;>\; d(x),
\end{align*}
which is a contradiction. Therefore (\ref{est_d_intersection}) cannot hold.

Hence (\ref{d_fixed_k}) holds for some $k \in \N$ that we fix for the remainder of the proof. For convenience, set $d' = (1+2^{-k}\ep)^{-1}d$. There are two cases to consider. First of all, it is possible that
\begin{equation}\label{eqn_case_1}
d' \in \bigcap_{n=1}^\infty \left(\cl{\bigcup\set{D_{k,F}}{F\subseteq \Gamma,\;|F| \in \N,\; |F|\geqslant n}}^{w^*}\right).
\end{equation}
In this case, by (\ref{defn_C_kF}) and (\ref{est_D_kF}), we have
\begin{equation}\label{d_prime_in_J_k}
d' \in \bigcap_{n=1}^\infty \left(\cl{\bigcup\set{C_{k,F}}{F\subseteq \Gamma,\;|F| \in \N,\; |F|\geqslant n}}^{w^*} + 3^{-k-n}\ep B_{X^*}\right) \;\subseteq\; \cl{J_k}^{w^*}.
\end{equation}
Now we need to consider whether $d$ (and hence $d'$) has infinite support or not. If $\supp(d)$ is infinite, then according to (\ref{defn_J_k}), Lemma \ref{lem_net_1} and the $w^*$-compactness of $C_k$, we have $d' \in C_k$. Then, using (\ref{defn_J_k}) and Lemma \ref{lem_convex_combination}, there exists $j \in J_k$ such that $d'(x) \leqslant j(x)$. As $F:=\supp(j)$ is finite, using (\ref{defn_C_kF}) and (\ref{est_D_kF})
\[
j \in C_{k,F} \;\subseteq\; \cl{D_{k,F}}^{w^*},
\]
and therefore, by (\ref{D_kF_boundary}), there exists $e' \in \cl{D_{k,F}}^{w^*} \cap A_F$ such that $j(x) \leqslant e'(x)$. Define $e=(1+2^{-k}\ep)e' \in \cl{D}^{w^*} \cap A_F \subseteq \cl{D}^{w^*} \cap A$. Evidently,
\[
\eta \;=\; d(x) \;\leqslant\; e(x) \;\leqslant\; \eta,
\]
meaning that we have found what we wanted. This concludes the case where $\supp(d)$ is infinite. If $G:=\supp(d)$ is finite, then from (\ref{defn_C_kF}) and (\ref{d_prime_in_J_k}) we see straightaway that $d' \in C_{k,G}$. Then we repeat what we have just done, but without $j$ (this time, by (\ref{D_kF_boundary}), there exists $e' \in \cl{D_{k,G}}^{w^*} \cap A_G$ such that $d'(x) \leqslant e'(x)$).

We have dealt with the first case, where (\ref{eqn_case_1}) holds. If (\ref{eqn_case_1}) does not hold then there exists some $n \in \N$ such that
\[
d' \in \cl{\bigcup\set{D_{k,F}}{F\subseteq \Gamma,\;|F|= n}}^{w^*}.
\]
We take nets $F^\alpha \subseteq \Gamma$, $|F^\alpha|=n$ and $d^\alpha \in D_{k,F^\alpha}$, such that $d^\alpha \stackrel{w^*}{\to} d'$. The set of subsets of $\Gamma$ having cardinality at most $n$ is compact in the topology of pointwise convergence on $\Gamma$. Among the accumulation points of the net $(F^\alpha)$, let $G$ be one having maximal cardinality. By taking a subnet if necessary, we can assume that $F^\alpha \to G$ and $G \subseteq F^\alpha$ for all $\alpha$. That $G$ is an accumulation point the $F^\alpha$ implies $\supp(d') \subseteq G$, so $d' \in W_G$. Either $|G|=n$ or $|G|<n$. If $|G|=n$, then $G=F^\alpha$ for all $\alpha$, giving $d' \in \cl{D_{k,G}}^{w^*}$. In this case we can finish the proof as above.

In the final part of the proof, we show that it is impossible for the strict inequality $|G|<n$ to hold. For a contradiction, suppose otherwise, and let $H^\alpha = F^\alpha \setminus G$. If $H$ is an accumulation point of $(H^\alpha)$ then $G \cup H$ is an accumulation point of $(F^\alpha)$. This forces $H=\varnothing$, by maximality of the cardinality of $G$. Hence, by compactness, $H^\alpha \to \varnothing$ in the topology of pointwise convergence. According to (\ref{est_D_kF}), there exists $c^\alpha \in C_{k,F^\alpha}$ such that $\n{d^\alpha - c^\alpha} \leqslant 3^{-k-n}\ep$. By taking a further subnet if necessary, we can assume that $c^\alpha \to c'$, where $c' \in \cl{J_k}^{w^*}$. Since $H^\alpha \to \varnothing$, it follows that $\supp(c') \subseteq G$, and thus $c' \in C_{k,G}$. Moreover, by $w^*$-lower semicontinuity of the dual norm, $\n{d'-c'} \leqslant 3^{-k-n}\ep$. Therefore, by this and (\ref{est_D_kF}),
\[
d' + {\ts \frac{1}{2}}\cdot 3^{-k-n}\ep B_{W_G} \;\subseteq\;  C_{k,G} + {\ts \frac{1}{2}}\cdot 3^{-k-n+1}\ep B_{W_G} \;\subseteq\; C_{k,G} + {\ts \frac{1}{2}}\cdot 3^{-k-|G|}\ep B_{W_G} \;\subseteq\; \cl{D_{k,G}}^{w^*}.
\]
It follows that
\[
d + {\ts \frac{1}{2}}\cdot(1+2^{-k}\ep) 3^{-k-n}\ep B_{W_G} \;\subseteq\; \cl{D}^{w^*},
\]
but this contradicts our initial assumption that $d(x)=\eta$.
\end{proof}

\begin{cor}\label{cor_main_tool} Let $(A_\gamma)_{\gamma \in \Gamma}$ be an admissible family as above, where each $A_\gamma$ is a $\sigma$-$w^*$-LRC and $w^*$-$K_\sigma$ set. Let $\ndot'$ be an equivalent norm on $X$ and let $B$ be a boundary of $\ndot'$, such that $\theta(f) < \infty$ whenever $f \in B$. Then $\ndot'$ can be approximated by both $C^\infty$-smooth norms and polyhedral norms.
\end{cor}

\begin{proof} Set $C=B_{(X,\ndot')^*}$. Given $\ep>0$, use Theorem \ref{thm_main_tool} to find a set $D \subseteq X^*$ such that
\[
B \;\subseteq\; \cl{D}^{w^*} \;\subseteq\; (1+\ep)C \qquad\text{and}\qquad \cl{D}^{w^*} \cap A \text{ is a boundary of }\cl{D}^{w^*},
\]
where $A:=\lspan(A_\gamma)_{\gamma \in \Gamma}$ is $\sigma$-$w^*$-LRC and $w^*$-$K_\sigma$ by Theorem \ref{thm_linspan}. Define $\tndot$ on $X$ by
\[
\tn{x} \;=\; \sup\set{|f(x)|}{f \in D}, \qquad\qquad x \in X.
\]
Then $\n{x}' \leqslant \tn{x} \leqslant (1+\ep)\n{x}'$ and $B_{(X,\ttrin)^*} = \cl{\conv}^{w^*}(D \cup (-D))$, and therefore
\[
(\cl{D}^{w^*} \cup (-\cl{D}^{w^*})) \cap A,
\]
is a boundary of $\tndot$. We finish the proof by appealing to Theorem \ref{thm_tb}.
\end{proof}

The last two results of the section will be of most use in the coming sections. The next corollary is immediate.

\begin{cor}\label{cor_main_tool_2} Let $(A_\gamma)_{\gamma \in \Gamma}$ be as above, and suppose that $\theta(f) < \infty$ for all $f \in X^*$. Then every equivalent norm on $X$ can be approximated by both $C^\infty$-smooth norms and polyhedral norms.
\end{cor}

We conclude the section by making one further reduction. Suppose that $P_\gamma X$ is 1-dimensional for all $\gamma \in \Gamma$, or equivalently, that we have a shrinking bounded M-basis $(e_\gamma,e^*_\gamma)_{\gamma \in \Gamma}$. Now we let $A_\gamma = P^*_\gamma X^* = \lspan(e^*_\gamma)$, which is always a $\sigma$-$w^*$-LRC and $w^*$-$K_\sigma$ set. Moreover, (\ref{defn_rho}) simplifies to
\begin{equation}\label{defn_rho_reduced}
\rho(F) \;=\; \max\set{\n{\sum_{\gamma \in F}a_\gamma e_\gamma^*}}{a_\gamma \in \R \text{ and }|a_\gamma|\|e^*_\gamma\|\leqslant 1 \text{ whenever }\gamma \in F}.
\end{equation}

\begin{cor}\label{cor_main_tool_3} Let $(e_\gamma,e^*_\gamma)_{\gamma \in \Gamma}$ be a shrinking bounded M-basis of $X$, and suppose that $\theta(f) < \infty$ for all $f \in X^*$. Then every equivalent norm on $X$ can be approximated by both $C^\infty$-smooth norms and polyhedral norms.
\end{cor}

\section{Spaces having a symmetric basis}\label{sect_symmetric}

The main purpose of this section is to develop tools to allow us to apply Corollary \ref{cor_main_tool_3} to certain spaces having symmetric bases. These tools focus on making it easier to establish whether or not $\theta(f)<\infty$ for all $f \in X^*$. Let $X$ be a Banach space and $\Gamma$ a set. Recall that a family of vectors $(e_\gamma)_{\gamma \in \Gamma}$ is a \emph{symmetric basis} of $X$ if, first, given $x \in X$, there is a unique family of scalars $(a_\gamma)_{\gamma \in \Gamma}$, such that $x = \sum_{\gamma \in \Gamma} a_\gamma e_\gamma$ (where the convergence is necessarily unconditional), and second, given any permutation $\pi$ of $\Gamma$, the sum $\sum_{\gamma \in \Gamma} b_\gamma e_{\pi(\gamma)}$ converges whenever $\sum_{\gamma \in \Gamma}b_\gamma e_\gamma$ converges (again, unconditionally). By the uniform boundedness principle, with such a basis in hand, we have
\[
K\;:=\; \sup_{\theta, \pi} \n{T_{\theta, \pi}} \;<\; \infty,
\]
where, given a permutation $\pi$ of $\Gamma$ and $\theta=(\theta_\gamma)_{\gamma \in \Gamma}$ a choice of signs, the operator $T_{\theta,\pi}$ is defined on $X$ by
\[
T_{\theta,\pi}\bigg( \sum_{\gamma \in \Gamma} a_\gamma e_\gamma\bigg) \;=\; \sum_{\gamma \in \Gamma} a_\gamma\theta_\gamma e_{\pi(\gamma)}.
\]
The number $K$ is known as the \emph{symmetric basis constant} of $(e_\gamma)_{\gamma \in \Gamma}$. If we define a new norm
\[
\pn{x}{s} \;=\; \sup_{\theta,\pi} \n{T_{\theta,\pi} x},
\]
then $\n{x} \leqslant \pn{x}{s} \leqslant K\n{x}$ for all $x \in X$, and the symmetric basis constant of $(e_\gamma)_{\gamma \in \Gamma}$ with respect to $\pndot{s}$ is equal to $1$ (see, for example \cite[Section 3.a]{lt:77}). Evidently, $\ndot$ and $\pndot{s}$ are equal if and only if $K=1$.

Hereafter, suppose that the symmetric basis $(e_\gamma)_{\gamma \in \Gamma}$ is shrinking. Let us fix a sequence $(\gamma_n)$ of distinct points in $\Gamma$. Following \cite[Proposition 3.a.6]{lt:77}, given $n \in \N$, define
\[
\lambda(n) \;=\; \pn{\sum_{k=1}^n e_{\gamma_k}}{s} \qquad\text{and}\qquad\mu(n) \;=\; \pn{\sum_{k=1}^n e^*_{\gamma_k}}{s},
\]
and set $\lambda(0)=\mu(0)=0$. Since the basis is $1$-symmetric with respect to $\pndot{s}$, the definitions of $\lambda(n)$ and $\mu(n)$ are independent of the choice of initial sequence $(\gamma_n)$. It is also important to note that, by \cite[Proposition 3.a.6]{lt:77}, we have $\lambda(n)\mu(n)=n$ for all $n$. By scaling the basis vectors by the same amount, we can assume that $\pn{e_\gamma}{s}=\|e^*_\gamma\|_s=1$ for all $\gamma \in \Gamma$.

The functions $\lambda$ and $\mu$ will help to simplify the task of establishing whether or not $\theta(f)$ is finite.

\begin{prop}\label{prop_rho_and_mu} Let $X$ have a shrinking symmetric basis as above. Then 
\[
\theta(f) \;<\; \infty \qquad\text{if and only if}\qquad \sum_{k=1}^\infty q_k(f)\mu(|G_k(f)|) \;<\; \infty.
\]
\end{prop}

\begin{proof} Let $K$ be the symmetric basis constant. Let $F\subseteq \Gamma$ be finite. The result will follow immediately if we can prove that $K^{-1}\mu(|F|) \leqslant \rho(F) \leqslant K\mu(|F|)$. First, the simplification of (\ref{defn_rho}) to (\ref{defn_rho_reduced}) applies here. Next, if we set $a_\gamma=K^{-1}$, $\gamma \in F$, then $|a_\gamma|\|e^*_\gamma\| \leqslant K^{-1}K\|e^*_\gamma\|_s=1$, and thus
\[
K^{-1}\mu(|F|) \;=\; K^{-1}\pn{\sum_{\gamma \in F}e_\gamma^*}{s} \;\leqslant\; \n{\sum_{\gamma \in F}a_\gamma e_\gamma^*} \;\leqslant\; \rho(F).
\]
On the other hand, if $|a_\gamma|\|e^*_\gamma\| \leqslant 1$, then $|a_\gamma| \leqslant |a_\gamma|\|e^*_\gamma\|_s \leqslant 1$. Therefore,
\[
\n{\sum_{\gamma \in F}a_\gamma e_\gamma^*} \;\leqslant\; K\pn{\sum_{\gamma \in F}a_\gamma e_\gamma^*}{s} \;\leqslant\; K \pn{\sum_{\gamma \in F} e_\gamma^*}{s} \;=\; K\mu(|F|). \qedhere
\]
\end{proof}

Recall the comments about subspaces of $c_0(\Gamma)$ following Theorem \ref{c0_approx}. The following result uses the function $\lambda$ to provide a straightforward test that will ensure that the examples we present are new.

\begin{prop}\label{prop_not_subspace_c_0} Let $X$ have a symmetric basis $(e_\gamma)_{\gamma\in\Gamma}$ as above. The following statements are equivalent. 
\begin{enumerate}
\item The sequence $(\lambda(n))_{n=1}^\infty$ is bounded.
\item The spaces $X$ and $c_0(\Gamma)$ are isomorphic.
\item There exists a set $\Delta$ such that $X$ is isomorphic to a subspace of $c_0(\Delta)$.
\end{enumerate}
\end{prop}

\begin{proof}The implication $(1) \Rightarrow (2)$ follows immediately from the fact that, given a normalized basis $(e_\gamma)_{\gamma \in \Gamma}$ of a Banach space having unconditional basis constant $L$, we have
\[
L^{-1}\max_{\gamma \in F} |a_\gamma| \;\leqslant\; \n{\sum_{\gamma \in F} a_\gamma e_\gamma} \;\leqslant\; L\max_{\gamma \in F} |a_\gamma|\n{\sum_{\gamma \in F} e_\gamma}.
\]
for every finite set $F \subseteq \Gamma$, and real numbers $a_\gamma$, $\gamma \in F$.

The implication $(2) \Rightarrow (3)$ is trivial. Finally, consider $(3) \Rightarrow (1)$. Assume that the bounded linear map $\map{T}{X}{c_0(\Delta)}$ is bounded below. Then $(Te_{\gamma_k})_{k=1}^\infty$ converges weakly to $0$ and satisfies $\inf (\pn{Te_{\gamma_k}}{\infty})_k > 0$, so it admits a subsequence, again labelled $(Te_{\gamma_k})$, that is equivalent to a block basic sequence of $c_0$ \cite[Corollary 4.27]{fhhspz:11}. Consequently, $(Te_{\gamma_k})$, and hence $(e_{\gamma_k})$, is equivalent to the usual basis of $c_0$ \cite[Proposition 4.45]{fhhspz:11}. The boundedness of $(\lambda(n))_{n=1}^\infty$ follows.
\end{proof}

We remark that if that $\Gamma$ is uncountable, then the implication $(3) \Rightarrow (2)$ above follows from \cite[Main Theorem]{ht:93}. There, it is shown that if $(u_\alpha)_{\alpha \in A}$ is an uncountable symmetric basic set in an F-space $Y$ having an F-norm and symmetric basis $(v_\beta)_{\beta \in B}$, then there exists a decreasing sequence of non-negative scalars $a_i$, $i \in \N$, and an injection $(\alpha,i) \mapsto \beta_{\alpha,i}$ from $A \times \N$ into $B$, such that $(u_\alpha)_{\alpha \in A}$ is equivalent to $(u'_\alpha)_{\alpha \in A}$, where $u'_\alpha:=\sum_{i=1}^\infty a_i v_{\beta_{\alpha,i}}$ converges in norm for all $\alpha \in A$. Therefore, if $Y=c_0(B)$ and $(v_\beta)_{\beta \in B}$ is its standard basis, then for every finite set $F \subseteq A$ and scalars $c_\alpha$, $\alpha \in F$, we have
\[
\pn{\sum_{\alpha \in F} c_\alpha u'_\alpha}{\infty} = \pn{\sum_{\alpha \in F}c_\alpha \sum_{i=1}^\infty a_i v_{\beta_{\alpha,i}}}{\infty} = \pn{\sum_{(\alpha,i) \in F \times \N} c_\alpha a_i v_{\beta_{\alpha,i}}}{\infty} \leqslant a_1 \max\set{|c_\alpha|}{\alpha \in F}.
\]

The next theorem is the main result of the section.

\begin{thm}\label{theta_equivalences} Let $X$ have a shrinking symmetric basis $(e_\gamma)_{\gamma \in \Gamma}$. The following statements are equivalent.
\begin{enumerate}
\item\label{finite_theta} $\theta(f)<\infty$ for all $f \in X^*$;
\item\label{eqn2} the quantity
\[
\sup\set{\n{\sum_{k=1}^n (\mu(k)-\mu(k-1))e_{\gamma_k}}}{n \in \N},
\]
is finite;
\item\label{bidual} the series
\[
\sum_{k=1}^\infty (\mu(k)-\mu(k-1))e_{\gamma_k},
\]
converges in $(X^{**},w^{**})$.
\end{enumerate}
\end{thm}

\begin{proof} The equivalence of (\ref{eqn2}) and (\ref{bidual}) follows because the basis is shrinking \cite[Proposition 1.b.2]{lt:77}. Now we prove the equivalence of (\ref{finite_theta}) and (\ref{eqn2}). Let $K$ be the symmetric basis constant of $(e_\gamma)_{\gamma \in\Gamma}$. As above, we assume that the basis is normalized with respect to $\pndot{s}$. Let $f \in X^*$. By Proposition \ref{prop_rho_and_mu}, 
\begin{equation}\label{eqn-1}
\theta(f) \;<\; \infty \qquad\text{if and only if}\qquad \sum_{k=1}^\infty q_k(f)\mu(|G_k(f)|) \;<\; \infty.
\end{equation}
By Lemma \ref{lem_ran(f)_structure}, there exists a sequence $(\gamma'_i)_{i=1}^\infty$ of distinct points in $\Gamma$, and integers $1=i_1 < i_2 < i_3 < \dots$, such that 
\[
G_k(f) \;=\; \set{\gamma'_i}{1 \leqslant i < i_{k+1}},
\]
whenever $k \in \N$ and $p_k(f)>0$. Define $a_i = \|P^*_{\gamma'_i} f\| = |f(e_{\gamma'_i})|\|e_{\gamma'_i}^* \|$, $i \in \N$. It follows that $p_k(f)=a_i$ whenever $i_k \leqslant i < i_{k+1}$. Now
\begin{align}\label{eqn0}
\sum_{i=1}^\infty (a_i-a_{i+1})\mu(i) \;&=\; \sum_{k=1}^\infty \sum_{i=i_k}^{i_{k+1}-1} (a_i-a_{i+1})\mu(i) \nonumber\\
&=\; \sum_{k=1}^\infty (a_{i_k}-a_{i_k+1})\mu(i_k) \;=\; \sum_{k=1}^\infty q_k(f)\mu(|G_k(f)|).
\end{align}
Observe that
\begin{equation}\label{eqn3}
\sum_{i=1}^n (a_i-a_{i+1})\mu(i) \;=\; \sum_{i=1}^n a_i(\mu(i)-\mu(i-1)) - a_{n+1}\mu(n).
\end{equation}
Since $\n{e_\gamma^*} \leqslant K\pn{e_\gamma^*}{s}=K$ for all $\gamma \in \Gamma$, and the basis is 1-symmetric with respect to $\pndot{s}$, we can see that
\[
a_{n+1}\mu(n) \;=\; \pn{\sum_{i=1}^n a_{n+1}e^*_{\gamma'_i}}{s} \;\leqslant\; \pn{\sum_{i=1}^n a_ie^*_{\gamma'_i}}{s} \;\leqslant\; K\pn{\sum_{i=1}^n \frac{a_i}{\|e^*_{\gamma'_i}\|}e^*_{\gamma'_i}}{s} \;\leqslant\; K\pn{f}{s},
\]
for all $n$. Therefore, by (\ref{eqn-1}) and (\ref{eqn0}), and using the fact that all the terms in the two partial sums in equation (\ref{eqn3}) are non-negative, $\theta(f)<\infty$ if and only if
\begin{equation}\label{eqn4}
\sum_{k=1}^\infty a_k(\mu(k)-\mu(k-1)) \;<\; \infty.
\end{equation}
By appealing to the symmetry of the basis, (\ref{eqn4}) holds for all $f \in X^*$ if and only if the family of vectors
\[
\set{\sum_{k=1}^n (\mu(k)-\mu(k-1))e_{\gamma_k}}{n \in \N},
\]
is weakly bounded, and hence norm bounded, by the uniform boundedness principle.
\end{proof}

Armed with Theorem \ref{theta_equivalences}, we turn to our first class of new examples.

\begin{example}\label{ex_Lorentz}
Consider a decreasing sequence $w=(w_n)_{n=1}^\infty$ of positive numbers such that $\sum_{n=1}^\infty w_n = \infty$. The {\em Lorentz space} $d(w,1,\Gamma)$ is the space of all functions $\map{f}{\Gamma}{\R}$, such that
\begin{equation}\label{defn_Lorentz_norm}
\n{f} \;:=\; \sup\set{\sum_{\gamma \in \Gamma} w_n|f(\gamma_n)|}{(\gamma_n)_{n=1}^\infty \subseteq \Gamma \text{ is a sequence of distinct points}},
\end{equation}
is finite. A treatment of the separable version of these spaces (where $\Gamma=\N$) can be found in \cite[Section 4.e]{lt:77}. It is clear that if $(\gamma_n)_{n=1}^\infty \subseteq \Gamma$ is chosen in such a way that $(|f(\gamma_n)|)_{n=1}^\infty$ is decreasing, then
\[
\n{f} \;=\; \sum_{n=1}^\infty w_n |f(\gamma_n)|.
\]
The {\em predual of Lorentz space}, $d_*(w,1,\Gamma)$, is the set of all functions $\map{x}{\Gamma}{\R}$, such that $\cl{x} \in c_0$, where
\[
\cl{x}(k) \;=\; \max\set{\frac{\sum_{i=1}^k |x(\gamma_i)|}{\sum_{i=1}^k w_i}}{\gamma_1,\ldots,\gamma_k \in \Gamma \text{ are distinct}},
\]
and $\n{x}=\pn{\cl{x}}{\infty}$. The separable version of these spaces was first considered in \cite{sargent:60}. The families of unit vectors $(e_\gamma)_{\gamma \in \Gamma}$ and $(e_\gamma^*)_{\gamma \in \Gamma}$ form canonical 1-symmetric bases of $d_*(w,1,\Gamma)$ and $d(w,1,\Gamma)$, respectively.

It is straightforward to see that
\begin{equation}\label{eqn_Lorentz_mu}
\mu(n) \;=\; \pn{\sum_{k=1}^n e^*_{\gamma_k}}{s} \;=\; \n{\sum_{k=1}^n e^*_{\gamma_k}} \;=\; \sum_{k=1}^n w_k,
\end{equation}
and that Theorem \ref{theta_equivalences} (2) is fulfilled trivially by any space $d_*(w,1,\Gamma)$, as
\[
\n{\sum_{k=1}^n (\mu(k)-\mu(k-1))e_{\gamma_k}} \;=\; \n{\sum_{k=1}^n w_k e_{\gamma_k}} \;=\; 1, 
\]
for all $n \in \N$. Hence Corollary \ref{cor_main_tool_3} applies. Finally,
\[
\lambda(n) \;=\; \frac{n}{\mu(n)} \;=\; \frac{n}{\sum_{k=1}^n w_k},
\]
and this forms a bounded sequence if and only if $w_n \not\to 0$. Therefore, provided $w_n \to 0$, Proposition \ref{prop_not_subspace_c_0} tells us that $d_*(w,1,\Gamma)$ is not isomorphic to a subspace of $c_0(\Delta)$, for any set $\Delta$.
\end{example}

The next result provides another test.

\begin{cor}\label{lambda_only} Let $X$ have a shrinking symmetric basis $(e_\gamma)_{\gamma \in \Gamma}$, and suppose that
\[
\sup\set{\n{\sum_{k=1}^n \frac{e_{\gamma_k}}{\lambda(k)}}}{n \in \N} \;<\; \infty,
\]
or equivalently, the series
\[
\sum_{k=1}^\infty \frac{e_{\gamma_k}}{\lambda(k)},
\]
converges in $(X^{**},w^{**})$. Then $\theta(f)<\infty$ for all $f \in X^*$.
\end{cor}

\begin{proof}
Because $\lambda(n)\mu(n)=n$ for all $n$,
\[
\mu(k) - \mu(k-1) \;=\; \frac{k}{\lambda(k)} - \frac{k-1}{\lambda(k-1)} \;\leqslant\; \frac{k}{\lambda(k)} - \frac{k-1}{\lambda(k)} \;=\; \frac{1}{\lambda(k)}.
\]
It follows that Theorem \ref{theta_equivalences} (2) is fulfilled.
\end{proof}

We provide an application of Corollary \ref{lambda_only}, by considering a symmetric version of the Nakano space. Let $\Gamma$ be a set and let $(p_k)_{k=1}^\infty$ be a non-decreasing sequence, with $p_1 \geqslant 1$. By $h^S_{(p_k)}(\Gamma)$ we denote the space of all real functions $x$ defined on $\Gamma$, such that
\[
\phi\bigg(\frac{x}{\rho}\bigg) \;<\; \infty,
\]
for all $\rho > 0$, where
\[
\phi(x) \;:=\; \sup\set{\sum_{k=1}^\infty |x(\gamma_k)|^{p_k}}{(\gamma_k)_{k=1}^\infty \text{ is a sequence of distinct points in }\Gamma}.
\]
Given $x \in h^S_{(p_n)}(\Gamma)$, we set
\[
\n{x} \;=\; \inf\set{\rho>0}{\phi\bigg(\frac{x}{\rho}\bigg) \leqslant 1}.
\]
It is easy to see that the standard unit vectors $(e_\gamma)_{\gamma \in \Gamma}$ form a $1$-symmetric basis in $h^S_{(p_n)}(\Gamma)$. In \cite{afnst:18}, it is shown that if $p_n \to \infty$, then $h^S_{(p_n)}(\Gamma)$ is isomorphically polyhedral.

\begin{prop}\label{prop_Nakano} Let $\sum_{k=1}^\infty k^{-1}\rho^{-p_k}$ converge for some $\rho>1$. Then
\[
\sup\set{\n{\sum_{k=1}^n \frac{e_{\gamma_k}}{\lambda(k)}}}{n \in \N} \;<\; \infty.
\]
\end{prop}

\begin{proof} Since $\lambda(1)=1$ and the sequences $(p_k)_{k=1}^\infty$ and $(\lambda(k))_{k=1}^\infty$ are non-decreasing, we know that
\[
n\lambda(n)^{-p_n} \;\leqslant\; \sum_{k=1}^n \lambda(n)^{-p_k} \;=\; \phi\left( \frac{\sum_{k=1}^n e_{\gamma_k}}{\lambda(n)} \right) \;=\; 1,
\]
hence
\begin{equation}\label{eqn_nakano_1}
\lambda(n)^{-p_n} \leqslant n^{-1},
\end{equation}
for all $n$. Using the hypothesis, let $m\in\N$ such that $\sum_{k=m+1}^\infty k^{-1} \rho^{-p_k} \leqslant 1$. It follows that 
\[
\sum_{k=m+1}^\infty \left(\frac{1}{\lambda(k)\rho}\right)^{p_k} \;\leqslant\; 1,
\]
and thus
\[
\n{\sum_{k=m+1}^n \frac{e_{\gamma_k}}{\lambda(k)}} \;\leqslant\; \rho,
\]
whenever $n \geqslant m+1$. We conclude that
\[
\n{\sum_{k=1}^n \frac{e_{\gamma_k}}{\lambda(k)}} \;\leqslant\; \n{\sum_{k=1}^m \frac{e_{\gamma_k}}{\lambda(k)}} + \rho,
\]
for all $n$.
\end{proof}

Corollary \ref{lambda_only} applies to any space $h^S_{(p_n)}(\Gamma)$ satisfying the hypothesis of Proposition \ref{prop_Nakano}. By Proposition \ref{prop_not_subspace_c_0}, $h^S_{(p_n)}(\Gamma)$ does not embed into any space of the form $c_0(\Delta)$, provided $\log(\lambda(n)) \to \infty$. Using (\ref{eqn_nakano_1}), we deduce that the same holds if
\begin{equation}\label{eqn_nakano_2}
\lim_{k \to\infty} \frac{\log(k)}{p_k} \;=\; \infty.
\end{equation}

\begin{example} Set $p_k = 2\log(\log(k)+1)+1$. Then
\[
\sum_{k=2}^\infty k^{-1}\me^{-p_k} \;<\; \sum_{k=2}^\infty \frac{1}{k\log^2(k)} \;<\; \infty,
\]
and thus Corollary \ref{lambda_only} applies to $h^S_{(p_n)}(\Gamma)$. Moreover, for large $k$, we have
\[
\frac{\log(k)}{p_k} \;=\; \frac{\log(k)}{2\log(\log(k)+1)+1} \;>\; \frac{\log(k)}{3\log(\log(k))} \;\to\; \infty,
\]
as $k \to \infty$.
\end{example}

\begin{prob} Does there exist a non-decreasing sequence $(p_k)_{k=1}^\infty$, with $p_1 \geqslant 1$, $p_k \to \infty$, such that 
\[
\sup\set{\n{\sum_{k=1}^n \frac{e_{\gamma_k}}{\lambda(k)}}}{n \in \N} \;=\; \infty,
\]
or
\[
\sup\set{\n{\sum_{k=1}^n (\mu(k)-\mu(k-1))e_{\gamma_k}}}{n \in \N} \;=\; \infty,
\]
with respect to $h^S_{(p_n)}(\Gamma)$?
\end{prob}

For some classes of spaces, for example, Orlicz spaces whose standard basis is shrinking, the implication in Corollary \ref{lambda_only} is reversible -- see Theorem \ref{Orliczcondition} below. We close the section by showing that this is not the case in general.

\begin{example} Let $w_n = n^{-1}$, $n \in \N$. Then $d_*(w,1):=d_*(w,1,\N)$ satisfies the conditions of Theorem \ref{theta_equivalences}, but not the hypothesis of Corollary \ref{lambda_only}.
\end{example}

\begin{proof} The first conclusion is shown in Example \ref{ex_Lorentz}. To see the second, by (\ref{eqn_Lorentz_mu}) we have 
\[
\log(n) \;\leqslant\; \mu(n) \;=\; \sum_{k=1}^n k^{-1} \;\leqslant\; 1 + \log(n),
\]
for all $n \in \N$, and hence, by the definition of the norm on $d_*(w,1)$,
\begin{align*}
\n{\sum_{k=1}^n \frac{e_k}{\lambda(k)}} \;\geqslant\; \frac{\sum_{k=1}^n \lambda(k)^{-1}}{\sum_{k=1}^n k^{-1}} &\;=\; \frac{\sum_{k=1}^n k^{-1}\mu(k)}{\mu(n)}\\
&\;\geqslant\; \frac{1}{1+\log(n)}\sum_{k=1}^n \frac{\log(k)}{k}.
\end{align*}
The function $\log(t)/t$ is decreasing for $t \geqslant \me$, thus given $n\geqslant 3$
\[
\sum_{k=3}^n \frac{\log(k)}{k} \;\geqslant\; \lint{3}{n+1}{\frac{\log(t)}{t}}{t} \;=\; {\ts \frac{1}{2}\big(\log^2(n+1)-\log^2(3)\big)}.
\]
Consequently,
\[
\n{\sum_{k=1}^n \frac{e_k}{\lambda(k)}} \;\geqslant\; \frac{\frac{1}{2}\big(\log^2(n) - \log^2(3)\big)}{1+\log(n)} \;\to\; \infty. \tag*{\qedhere}
\]
\end{proof}

\section{Orlicz spaces and Leung's Condition}

In this section, we consider Orlicz space in the context of norm approximation. Let $M$ be an Orlicz function and let $\Gamma$ be a set. The space $\ell_M(\Gamma)$ is the set of all functions $\map{x}{\Gamma}{\R}$ such that
\[
\n{x} :\;=\; \inf\set{\rho>0}{\sum_{\gamma \in \Gamma} M\bigg(\frac{|x(\gamma)|}{\rho}\bigg) \leqslant 1},
\]
is finite. The space $h_M(\Gamma)$ is that closed subspace of $\ell_M(\Gamma)$ for which 
\[
\sum_{\gamma \in \Gamma} M\bigg(\frac{|x(\gamma)|}{\rho}\bigg) \;<\; \infty,
\]
for \emph{all} $\rho>0$. We denote by $h_M$ the space $h_M(\N)$. It is easy to check that the standard unit vectors $(e_\gamma)_{\gamma \in \Gamma}$ form a 1-symmetric basis of $h_M(\Gamma)$. It follows from \cite[Proposition 1.b.2]{lt:77} and the definitions that if this basis is shrinking (and in this section we will always assume so), then $\ell_M(\Gamma)$ is isometric to $h_M(\Gamma)^{**}$. We assume hereafter that $M$ is \emph{non-degenerate}, that is, $M(t)>0$ for all $t>0$. In this case, basic calculation yields
\begin{equation}\label{eqn_Orlicz_lambda}
\lambda(n) \;=\; \n{\sum_{k=1}^n e_{\gamma_k}} \;=\; \frac{1}{M^{-1}\big(\frac{1}{n}\big)},
\end{equation}
and thus $h_M(\Gamma)$ is not isomorphic to a subspace of $c_0(\Delta)$, for any $\Delta$, by Proposition \ref{prop_not_subspace_c_0} (if $M$ is degenerate then $h_M(\Gamma)$ is isomorphic to $c_0(\Gamma)$). 

The theory of polyhedrality in Orlicz sequence spaces was initiated by Leung in \cite{leung:94,leung:99}. His work focuses in large measure on a condition on $M$ that we shall call \emph{Leung's Condition}.

\begin{defn}[cf.~{\cite[Theorem 4]{leung:94}}]\label{defn_Leung}
We say that a non-degenerate Orlicz function $M$ satisfies {\em Leung's Condition} if there exists $K>1$ such that
\[
\lim_{t \to 0} \frac{M(K^{-1}t)}{M(t)} \;=\; 0.
\]
\end{defn}

The utility of this condition is demonstrated by the following results.

\begin{thm}[{\cite[Theorem 18]{leung:99}}]\label{theorem_Leung} The following statements are equivalent.
\begin{enumerate}
\item $M$ satisfies Leung's condition;
\item $h_M$ embeds isomorphically in $C(\omega^\omega+1)$;
\item $h_M$ embeds isomorphically in $C(\alpha+1)$ for some countable ordinal $\alpha$.
\end{enumerate}
\end{thm}

In particular, if $M$ satisfies Leung's condition then $h_M$ admits an equivalent polyhedral norm. This was first proved in \cite[Theorem 4]{leung:94}, but it follows easily from Theorem \ref{theorem_Leung} since, given a countable compact Hausdorff space $K$, the space $C(K)$ admits a countable boundary and any such space admits an equivalent polyhedral norm \cite[Theorem 3]{fonf:80}. More generally, Leung's condition implies the existence of an equivalent polyhedral norm on $h_M(\Gamma)$ \cite[Corollary 25]{fpst:08}. Concerning the approximation of all equivalent norms on $h_M(\Gamma)$ by polyhedral and $C^\infty$-smooth norms, we have the following result.

\begin{thm}\label{Orliczcondition}
Let $\Gamma$ be a set and let $M$ be a non-degenerate Orlicz function. The following statements are equivalent.
\begin{enumerate}
\item There exists $K>1$ such that
\begin{equation}\label{eqn5}
\sum_{n=1}^\infty M\bigg(\frac{M^{-1}\big(\frac{1}{n}\big)}{K} \bigg) \;<\; \infty.
\end{equation}
\item The quantity
\[
\sup\set{\n{\sum_{k=1}^n \frac{e_{\gamma_k}}{\lambda(k)}}}{n \in \N},
\]
is finite.
\item The standard basis of $X:=h_M(\Gamma)$ is shrinking and $\theta(f)<\infty$ for all $f \in X^*$.
\end{enumerate}
\end{thm}

Before proving the theorem, we present three results, mostly due to D.~Leung (who contacted us after seeing a previous version of the paper), that help us to compare condition (\ref{eqn5}) with Leung's condition. Only the final consequence of Proposition \ref{prop_implies_Leung} had been proved by us (in a more complicated way) before Leung made contact. These results have much enhanced this part of the paper and we are grateful to Leung for giving us permission to include them here.

\begin{lem}[\cite{leung:18}] Condition (\ref{eqn5}) in Theorem \ref{Orliczcondition} is equivalent to the existence of $K>1$ such that
\begin{equation}\label{Leung_sum}
\sum_{j=1}^\infty \frac{M(K^{-1}2^{-j})}{M(2^{-j})} \;<\; \infty.
\end{equation}
\end{lem}

\begin{proof} Given $j \in \N$, let $n_j \in \N$ be minimal, subject to the condition $n_j^{-1} \leqslant M(2^{-j})$. By this minimality, and the convexity of $M$, it follows that
\begin{equation}\label{Leung_sum_1}
2(n_j-1) \;<\; \frac{2}{M(2^{-j})} \;\leqslant\; \frac{1}{M(2^{-j-1})} \;\leqslant\; n_{j+1} \;<\; \frac{1}{M(2^{-j-1})} + 1.
\end{equation}
Given $m,n \in \N$, with $n_j \leqslant n < n_{j+1}$, we have $2^{-j-1} < M^{-1}(\frac{1}{n}) \leqslant 2^{-j}$, giving
\[
{\ts M(2^{-j-m-1}) \;<\; M(2^{-m}M^{-1}(\frac{1}{n})) \;\leqslant\; M(2^{-j-m}),}
\] 
and consequently
\[
\sum_{j=1}^\infty (n_{j+1}-n_j)M(2^{-j-m-1}) \;<\; \sum_{n=n_1}^\infty {\ts M(2^{-m}M^{-1}(\frac{1}{n}))} \;\leqslant\; \sum_{j=1}^\infty (n_{j+1}-n_j)M(2^{-j-m}).
\]
Hence the convergence of $\sum_{j=1}^\infty n_{j+1} M(2^{-j-m})$ implies that of $\sum_{n=1}^\infty M(2^{-m}M^{-1}(\frac{1}{n}))$. Together with (\ref{Leung_sum_1}), this in turn implies the convergence of $\sum_{j=1}^\infty n_{j+1} M(2^{-j-m-1})$. Again, given (\ref{Leung_sum_1}), the statements
\[
\sum_{j=1}^\infty n_{j+1} M(2^{-j-m}) \;<\; \infty \qquad\text{and}\qquad \sum_{j=1}^\infty \frac{M(2^{-j-m})}{M(2^{-j})} \;<\; \infty,
\]
are equivalent. This completes the proof. 
\end{proof}

\begin{prop}[mainly \cite{leung:18}]\label{prop_implies_Leung}Leung's condition is satisfied if and only if
\begin{equation}\label{Leung_sum_2}
\lim_{j \to \infty} \frac{M(2^{-j-m})}{M(2^{-j})} \;=\; 0,
\end{equation}
for sufficiently large $m \in \N$. Consequently, if $M$ satisfies condition (\ref{eqn5}) then it satisfies Leung's condition.
\end{prop}

\begin{proof}Clearly, if $M$ satisfies Definition \ref{defn_Leung} then it satisfies (\ref{Leung_sum_2}) for a sufficiently large $m$.
To see that the converse holds, let $2^{-j-1} \leqslant t \leqslant 2^{-j}$. Then $M(2^{-j-1}) \leqslant M(t)$ and $M(2^{-m-1}t) \leqslant M(2^{-m-1-j})$. Therefore
\[
\frac{M(2^{-m-1}t)}{M(t)} \;\leqslant\; \frac{M(2^{-j-1-m})}{M(2^{-j-1})},
\]
meaning that we obtain Leung's condition, where $K=2^{m+1}$. The result now clearly follows from Lemma \ref{Leung_sum}.
\end{proof}

On the other hand, Leung's condition does not imply condition (\ref{eqn5}).

\begin{example}[\cite{leung:18}] Let $M$ be an Orlicz function satisfying
\[
M'(t) \;=\; a_j \;:=\; \prod_{k=1}^j \frac{1}{\log(k+2)},
\]
whenever $2^{-j-1} < t < 2^{-j}$. Then $M$ satisfies Leung's condition but not condition (\ref{eqn5}).
\end{example}

\begin{proof} We have $M(2^{-j}) \;=\; \sum_{\ell=j}^\infty 2^{-\ell-1} a_\ell$, so $2^{-j-1}a_j \leqslant M(2^{-j}) \leqslant 2^{-j}a_j$. It follows that 
\[
\frac{M(2^{-j-1})}{M(2^{-j})} \;\leqslant\; \frac{a_{j+1}}{a_j} \;\to\; 0,
\]
as $j\to\infty$. However, given $m \in \N$,
\[
\sum_{j=1}^\infty \frac{M(2^{-j-m})}{M(2^{-j})} \;\geqslant\; \sum_{j=1}^\infty \frac{2^{-j-m-1} a_{j+m}}{2^{-j}a_j} \;=\; 2^{-m-1}\sum_{j=1}^\infty \frac{a_{j+m}}{a_j} \;\geqslant\; 2^{-m-1}\sum_{j=1}^\infty \frac{1}{\log^m(j+3)},
\]
which diverges. The result now follows from the previous two.
\end{proof}

\begin{proof}[Proof of Theorem \ref{Orliczcondition}]
First, we demonstrate the equivalence of (1) and (2). Let $(\gamma_n) \subseteq \Gamma$ be the sequence of distinct points in $\Gamma$ from Section \ref{sect_symmetric}. Suppose that (1) holds. Recall equation (\ref{eqn_Orlicz_lambda}). Set $L= \max\{1,\sum_{n=1}^\infty M\big(M^{-1}(\frac{1}{n})/K \big)\}$. By convexity of $M$ and the fact that $M(0)=0$, we have $\sum_{n=1}^\infty M\big(M^{-1}(\frac{1}{n})/KL \big) \leqslant 1$ and thus 
\[
\n{\sum_{k=1}^n \frac{e_{\gamma_k}}{\lambda(k)}} \;\leqslant\; KL,
\]
for all $n$, giving (2). The converse implication follows similarly. 

Now we prove that (2) implies (3). Let (2) hold. Since (2) implies (1), using Proposition \ref{prop_implies_Leung}, $M$ satisfies Leung's condition. This implies that $X$ cannot contain an isomorphic copy of $\ell_p$, $p \geqslant 1$ \cite[Theorem 4.a.9]{lt:77}. Since $X$ cannot contain an isomorphic copy of $\ell_1$, the standard basis of $X$ must be shrinking \cite[Theorem 1.c.9]{lt:77}. Now we are in a position to apply Corollary \ref{lambda_only}, giving (3).

Finally, we show that (3) implies (2). Observe that both conditions (2) and (3) hold independently of the choice of (equivalent) norm on $X$ that is used to define $\lambda$ and $\theta$:~if (2) is satisfied with respect to one equivalent norm, then it is satisfied with respect to all others, and likewise for (3). Since the basis of $X$ is shrinking, $X$ cannot contain an isomorphic copy of $\lp{1}$. Therefore, according to \cite[Theorem 4.a.9 and p.~144]{lt:77}, we can deduce that there exists $a>1$ and a differentiable Orlicz function $N$ that is equivalent to $M$ at $0$ and satisfies
\[
\lim_{t \to 0} \inf \frac{tN'(t)}{N(t)} \;>\; a.
\]
We extend the function $N$ for large $t$ in such a way that 
\[
\frac{tN'(t)}{N(t)} \;>\; a,
\]
for all $t>0$. Given $w >0$, define $F(w)=wt(w)$, where $t(w)=N^{-1}(\frac{1}{w})$. By the Inverse Function Theorem,
\begin{align*}
F'(w) \;=\; t(w) + wt'(w) &\;=\; t(w) + w\cdot\frac{1}{N'(t(w))}\cdot\bigg({-\frac{1}{w^2}}\bigg) \\
 &\;=\; t(w) - \frac{1}{wN'(t(w))}\\
 &\;>\; t(w) - \frac{t}{awN(t(w))} \;=\; \bigg(\frac{a-1}{a}\bigg)t(w) \;=\;  \bigg(\frac{a-1}{a}\bigg){\ts N^{-1}\big(\frac{1}{w}\big)}.
\end{align*}
Set $c = (a-1)/a \in (0,1)$. By the Mean Value Theorem, for each $n$ there exists $\tau_n \in (0,1)$ such that
\begin{equation}\label{eqn6}
F(n) - F(n-1) \;=\; F'(n-\tau_n) \;>\;  cN^{-1}\bigg(\frac{1}{n-\tau_n}\bigg) \;>\; c{\ts N^{-1}\big(\frac{1}{n}\big)}. 
\end{equation}

The function $N$ yields an equivalent 1-symmetric norm $\pndot{N}$ on $X$, with respect to which the basis is normalized. Let $\lambda_N$ and $\mu_N$ be the functionals corresponding to $\pndot{N}$. Condition (3) is satisfied with respect to $\pndot{N}$, and hence
\[
\sup\set{\pn{\sum_{k=1}^n (\mu_N(k)-\mu_N(k-1))e_{\gamma_k}}{N}}{n \in \N} \;<\; \infty,
\]
by Theorem \ref{theta_equivalences}. Since $\lambda_N(n) = N^{-1}(\frac{1}{n})$ and $\mu_N(n)=n/\lambda_N(n) = F(n)$, equation (\ref{eqn6}) means that
\[
\mu_N(n)-\mu_N(n-1) \;>\; \frac{c}{\lambda_N(n)},
\]
for all $n$. Therefore,
\[
\sup\set{\pn{\sum_{k=1}^n \frac{e_{\gamma_k}}{\lambda_N(k)}}{N}}{n \in \N} \;<\; \infty,
\]
from which condition (2) follows.
\end{proof}

\begin{example} The Orlicz function
\[
M(t) \;=\; \begin{cases}\me^{-\frac{1}{t}} & 0 < t \leqslant \frac{1}{2} \\ 0 & t=0,\end{cases}
\]
(and extended suitably for $t>\frac{1}{2}$) satisfies the condition in Theorem \ref{Orliczcondition}. Given $K>1$, we have
\[
M\bigg(\frac{M^{-1}\big(\frac{1}{n}\big)}{K} \bigg) \;=\; \me^{-K\log(n)} \;=\; n^{-K}. \tag*{\qedsymbol}
\]
\end{example}

\section{A class of $C(K)$ spaces having approximable norms}

It is well known that if $\Gamma$ is infinite then $c_0(\Gamma)$ is isomorphic to $C(\Gamma \cup \{\infty\})$, where $\Gamma \cup \{\infty\}$ is the 1-point compactification of $\Gamma$ endowed with the discrete topology. In this section, we use Corollary \ref{cor_main_tool_2} to present a class of compact scattered (Hausdorff) spaces $K$ having the property that any equivalent norm on $C(K)$ can be approximated by both $C^\infty$-smooth norms and polyhedral norms. This class contains, for every ordinal $\alpha$, a space $K$ such that the Cantor-Bendixson derivative $K^{(\alpha)}$ of order $\alpha$ is non-empty. This result was motivated by the following problem, which remains open (see \cite{hp:14} for related results).

\begin{prob}\label{prob_density_omega_1} Let $\alpha \geqslant \omega_1$. Can every equivalent norm on $C(\alpha+1)$ be approximated by $C^1$-smooth norms?
\end{prob}

It will be more convenient for us to work with locally compact scattered spaces in the main. Given such a space $M$, we identify the dual space $C_0(M)^*$ with $\ell_1(M)$ in its natural norm $\pndot{1}$. We define a \emph{tree} to be a partially ordered set, such that the set of predecessors of each element of the tree is well-ordered. Our class consists of trees that are locally compact with respect to a certain topology. Given an ordinal $\eta$, let $q(\eta)$ denote the unique ordinal having the property that
\[
\omega q(\eta) \;\leqslant\; \eta \;<\; \omega(q(\eta)+1).
\]
Given ordinals $\alpha$ and $\eta$, define the sets
\[
M_\alpha \;=\; \set{\map{t}{n}{\omega\alpha}}{\text{$1\leqslant n < \omega$ and $t(i+1)<\omega q(t(i))$ whenever $i+1<n$}},
\]
and
\[
K_\eta \;=\; \set{(\eta)^\frown t}{t \in M_{q(\eta)} \cup \{\varnothing\}} \;\subseteq\; M_{q(\eta)+1},
\]
where $n<\omega$ is treated in the definition of $M_\alpha$ as the set of ordinals strictly preceeding $n$, and where $^\frown$ denotes concatenation of sequences. We make $M_\alpha$ into a tree by partially ordering it with respect to end-extension:~$s \preccurlyeq t$ if and only if $|s| \leqslant |t|$ and $t\restrict{|s|}\, = s$, where $|t| := \dom t$. Given $t \in M_\alpha$, the condition $t(i+1)<\omega q(t(i))$ whenever $i+1<|t|$ implies that $t$ is a strictly decreasing sequence of ordinal numbers. It follows that $M_\alpha$ is \emph{well-founded}, that is, it does not admit any infinite totally ordered subsets.

We proceed to equip $M_\alpha$ with a topology. Given $s \in M_\alpha$ and a finite set $F \subseteq s(|s|-1)$, we define the set
\[
U_{s,F} \;=\; \{s\} \cup \set{t \in M_\alpha}{s \prec t \text{ and }t(|s|) \notin F},
\]
(here, as above, $s(|s|-1)$ is treated as the set of ordinals strictly preceeding the number itself). The family $\mathscr{U}_s :=\set{U_{s,F}}{F \subseteq s(|s|-1) \text{ is finite}}$ will be a local base of neighbourhoods of $s$, and $\bigcup_{s \in M_{\alpha}} \mathscr{U}_s$ a base of our topology on $M_\alpha$. This topology agrees with the so-called \emph{coarse wedge topology}, that can be defined on an arbitrary tree. Using the Alexander Subbase Theorem, it can be shown that $M_\alpha$ is locally compact with respect to this topology \cite{n:97}. Moreover, if $\eta<\omega\alpha$, then $K_\eta$ is a compact open subtree of $M_\alpha$.

In order to eliminate a potential source of confusion, we should point out that the coarse wedge topology is in general strictly finer than another (quite commonly used) locally compact topology with which one can endow a tree, namely the \emph{interval topology}. This topology has been used in a number of results in renorming theory, including \cite[Theorem 10]{fpst:08}, which characterises exactly when an equivalent polyhedral norm exists on $C_0(T)$, where $T$ is a tree. Since the coarse wedge topology is different, that result is not comparable with the work contained in this section.

Recall that, given a scattered locally compact space $M$, the \emph{scattered height} of $M$ is the least ordinal $\Omega$ for which $K^{(\Omega)}$ is empty. It is easy to compute the Cantor-Bendixson derivatives of the spaces $M_\alpha$.

\begin{prop}Given an ordinal $\xi$, we have
\[
M_\alpha^{(\xi)} \;=\; \set{t \in M_\alpha}{\xi \leqslant q(t(|t|-1))}.
\]
Consequently, the scattered height of $M_\alpha$ equals $\alpha$.
\end{prop}

\begin{proof} We proceed by transfinite induction on $\xi$. The result is obvious if $\xi=0$. Suppose that it holds for $\xi$. Let $t \in M_\alpha^{(\xi)}$. If $\xi = q(t(|t|-1))$ then $t$ must be maximal in $M_\alpha^{(\xi)}$, because it is impossible to strictly extend $t$ to $u \in M_\alpha^{(\xi)}$ in such a way that $u(|t|) < \omega q(u(|t|-1)) = \omega \xi$. Instead, if $\xi+1 \leqslant q(t(|t|-1))$ then $t^\frown(\omega\xi+n)$, $n < \omega$, are all strict extensions of $t$ in $M_\alpha^{(\xi)}$, and thus $t \in M_\alpha^{(\xi+1)}$. The result is therefore true for $\xi+1$. The limit ordinal case follows easily.
\end{proof}

By \cite[Theorem 3.8]{lancien:95}, given a locally compact space $M$, if $C_0(M)$ is isomorphic to a subspace of $c_0(\Delta)$ for some set $\Delta$, then the scattered height of $M$ equals $n$ for some $n<\omega$. Therefore, provided $\alpha \geqslant \omega$, $C_0(M_\alpha)$ cannot embed isomorphically into any such $c_0(\Delta)$.

Given a locally compact space $M$, let $A(M)$ be the subspace $\lspan\set{\delta_t}{t \in M}\subseteq C_0(M)^*$ of finite linear combinations of Dirac functionals on $M$. All $w^*$-relatively discrete subsets of dual spaces are $w^*$-LRC \cite[Example 6 (1)]{fpst:14}. Therefore, if $M$ is a \emph{$\sigma$-discrete} set, that is, the countable union of countably many relatively discrete subsets, then $\set{\delta_t}{t \in M} \cup \{0\}$ is a $w^*$-compact and $\sigma$-$w^*$-LRC subset of $C_0(M)^*$. Consequently, $A(M)$ is $\sigma$-$w^*$-compact and $\sigma$-$w^*$-LRC, by Theorem \ref{thm_linspan}. It is easy to see that the set of elements of $M_\alpha$ having length $n$ is relatively discrete, thus each $M_\alpha$ is $\sigma$-discrete and so $A(M_\alpha)$ is $\sigma$-$w^*$-LRC and $w^*$-$K_\sigma$. 

\begin{thm}\label{thm_tree}
Let $\alpha$ be an ordinal. Then every equivalent norm on $C_0(M_\alpha)$ can be approximated by both $C^\infty$-smooth norms and polyhedral norms.
\end{thm}

We set up some more machinery in order to prove Theorem \ref{thm_tree}. We will need to consider some isomorphisms. The isomorphisms stated in the result below are well known. We only include an explicit isomorphism in the proof in order to establish the additional connection between the subspaces of the duals.

\begin{prop}\label{prop_unif_bdd} Let $K$ be a compact space admitting a convergent sequence of distinct points. Then given $u \in K$, there exists a surjective isomorphism $\map{S}{C(K)}{C_0(K\setminus\{u\})}$, such that $S^*A(K\setminus\{u\})=A(K)$.
\end{prop}

\begin{proof} Let  $(t_n)_{n=1}^\infty$ be a convergent sequence of distinct points and let $u \in K$. Without loss of generality, we assume that $u \neq t_n$ for all $n\in\N$. The space $C_0(K\setminus\{u\})$ identifies isometrically with the hyperplane
\[
X \;=\; \set{f \in C(K)}{f(u)=0},
\]
which we will work with instead. Let $U_n$, $n \in \N$, be a collection of pairwise disjoint open sets such that $t_n \in U_n$ and $u \notin U_n$ for all $n$, and let $\map{\phi_n}{K}{[0,1]}$ be continuous functions satisfying $\phi_n(t_n)=1$ and $\phi_n(s)=0$ whenever $s \in K\setminus U_n$. We have $\phi_n \in X$ for all $n$. Define $\map{S}{C(K)}{X}$ by
\[
Sf \;=\; f - f(u)\mathbf{1}_K - f(t_0)\phi_0 + \sum_{n=1}^\infty (f(t_{n-1})-f(t_n))\phi_n.
\]
The infinite sum is an element of $X$ because the $U_n$ are pairwise disjoint and $f(t_{n-1})-f(t_n) \to 0$ as $n \to \infty$. We see that, given $\delta_s \in X^*$, $s \in K\setminus\{u\}$,
\[
S^*\delta_s \;=\; \delta_s - \delta_u - \phi_0(s)\delta_{t_0} + \sum_{n=1}^\infty \phi_n(s)(\delta_{t_{n-1}}-\delta_{t_n}).
\]
The apparently infinite sum is really a finite sum, because $\phi_n(s)$ can be non-zero for at most one $n$. Hence $S^*A(K\setminus\{u\}) \subseteq A(K)$.

We have $Sf(t_0)=-f(u)$ and $Sf(t_n) = f(t_{n-1})$ whenever $n \geqslant 1$. Using these facts, it is easy to verify that $S^{-1}$ exists and equals $\map{R}{X}{C(K)}$, where
\[
Rg \;=\; g-g(t_0)\mathbf{1}_K + g(t_1)\phi_0 + \sum_{n=1}^\infty (g(t_{n+1})-g(t_n))\phi_n.
\]
Likewise, it can be shown that $R^*\delta_s$ is finitely supported for all $s \in K$, so $R^*A(K) \subseteq A(K\setminus\{u\})$.
\end{proof}

Given an infinite compact scattered space $K$, it is obvious that there exists a sequence of isolated points that converges to an element of the first derivative $K'$. Hence the proposition above applies to any such $K$.

Fix an ordinal $\alpha\geqslant 1$. We wish to apply Corollary \ref{cor_main_tool_2} to the space $C_0(M_\alpha)$. This means that we need to build an appropriate system of projections, an admissible family of subspaces (in the sense of Definition \ref{defn_admissible_set}), and ensure that the corresponding function $\theta$ is always finite.

We begin by setting up the projections. Given a non-empty clopen (that is, closed and open) subset $L$ of a locally compact space $M$, we define a projection $P_L$ on $C_0(M)$ by $P_Lf = f\cdot \mathbf{1}_L$, where $\mathbf{1}_L$ is the characteristic function of $L$. It could be that $M$ is a clopen subset of another locally compact space $M'$, and if so, we will use the same notation $P_L$ for the corresponding projections on $C_0(M)$ and $C_0(M')$.

It is evident that $M_\alpha$ is the discrete union of pairwise disjoint compact open sets
\[
M_\alpha \;=\; \bigcup_{\eta < \omega\alpha} K_\eta,
\] 
meaning that $C_0(M_\alpha)$ is naturally isometric to the $c_0$-sum of the spaces $C(K_\eta)$, $\eta<\omega\alpha$. Given $\eta<\omega\alpha$, let $P_\eta = P_{K_\eta}$. Having in mind the decomposition of $M_\alpha$, it is clear that $(P_\eta)_{\eta < \omega\alpha}$ fulfills properties (1)--(4) listed at the start of Section \ref{sect_approx_frame}. We will take advantage of the fact that the images $P_\eta C_0(M_\alpha)$ and $P_\eta^* C_0(M_\alpha)^*$ are naturally isometric to $C(K_\eta)$ and its dual, respectively.

Next, we determine $\theta$ and establish that it is always finite. Let $\mu \in C_0(M_\alpha)^*$. Given finite $F \subseteq \omega\alpha$, the definition of $\rho$ in (\ref{defn_rho}) becomes
\[
\rho(F) \;=\; |F|,
\]
because $\pn{\sum_{\eta \in F}\mu_\eta}{1}=\sum_{\eta \in F}\pn{\mu_\eta}{1}$. Moreover, in this case (\ref{eqn_h_less_than_theta}) becomes
\[
\pn{h_m(\mu)}{1} \;=\; \sum_{k=1}^m q_k(\mu)\pn{\omega_k(\mu)}{1} \;=\; \sum_{k=1}^m q_k(\mu)|G_k(\mu)| \;=\; \sum_{k=1}^m q_k(\mu)\rho(G_k(\mu)).
\]
By Lemma \ref{lem_theta_dominate}, the left hand side converges to $\pn{\mu}{1}$, which means that $\theta(\mu) = \pn{\mu}{1}$ and is therefore always finite, regardless of the system of projections with respect to which it is defined. For this reason, we shall remove $\theta$ from all subsequent arguments.

Finally, we define a family of subspaces that we claim is admissible. Observe that
\[
A(K_\eta) \;\subseteq\; \cl{\lspan}^{\ndot}(\delta_t)_{t \in K_\eta}\;=\; C(K_\eta)^* \;\equiv\; P^*_\eta C_0(M_\alpha)^*,
\]
whenever $\eta<\omega\alpha$. The family we consider is $(A(K_\eta))_{\eta<\omega\alpha}$.

\begin{proof}[Proof of Theorem \ref{thm_tree}] We use transfinite induction on $\alpha\geqslant 1$ to show that $(A(K_\eta))_{\eta<\omega\alpha}$ is an admissible family on $C_0(M_\alpha)$. This is trivial in the case $\alpha=1$, because $K_\eta$ is the singleton $\{(\eta)\}$ and $A(K_\eta)=P^*_\eta(M_1)^*$ whenever $\eta<\omega$. Limit cases are straightforward. Assume that $\alpha$ is a limit ordinal and that $(A(K_\eta))_{\eta<\omega\xi}$ is admissible on $C_0(M_\xi)$ whenever $1 \leqslant \xi < \alpha$. Given a finite set $F \subseteq \omega\alpha$, there is an ordinal $\xi$ such that $1\leqslant \xi < \alpha$ and $F \subseteq \omega\xi$. Hence we can apply the admissibility of $(A(K_\eta))_{\eta<\omega\xi}$ on $C_0(M_\xi)$, which embeds naturally inside $C_0(M_\alpha)$, to see that the conditions in Definition \ref{defn_admissible_set} are fulfilled.

The successor case requires more work. Assume that $\alpha\geqslant 1$ and that $(A(K_\eta))_{\eta<\omega\alpha}$ is admissible on $C_0(M_\alpha)$. Fix a finite set $F \subseteq \omega(\alpha+1)$. There is nothing to prove if $F\subseteq \omega\alpha$, so we assume that $F\setminus \omega\alpha$ is non-empty. Define the subspaces
\begin{align*}
V\;&=\; \lspan(P_\lambda C_0(M_{\alpha+1}))_{\lambda \in F} \;\equiv\; C\bigg(\bigcup_{\lambda \in F} K_\lambda\bigg),\\
W\;&=\; \lspan(P^*_\lambda C_0(M_{\alpha+1})^*)_{\lambda \in F}, \text{ and}\\
A\;&=\;  \lspan(A(K_\lambda))_{\lambda \in F} \;=\; A\bigg(\bigcup_{\lambda \in F} K_\lambda\bigg).
\end{align*}
It is clear that $W$ is naturally isometric to $V^*$, so we regard them as equal.

Let $\ep>0$, and let $C \subseteq W$ be a $w^*$-compact set that is norming with respect to $V$. In order to get the set $D$ we need to fulfil Definition \ref{defn_admissible_set}, we will define a surjective isomorphism $\map{T}{V}{C_0(M_\alpha)}$ such that
\begin{equation}\label{eqn_C(K)_condition}
T^*A(M_\alpha) \;=\; A\bigg(\bigcup_{\lambda \in F} K_\lambda\bigg) \;=\; A.
\end{equation}
Then $(T^*)^{-1}C \subseteq C_0(M_\alpha)^*$ is $w^*$-compact and norming. As $(A(K_\eta))_{\eta<\omega\alpha}$ is admissible, by Theorem \ref{thm_main_tool}, there exists $D \subseteq W$ such that 
\[
(T^*)^{-1}C \;\subseteq\; \cl{(T^*)^{-1} D}^{w^*} \subseteq (T^*)^{-1}C + \|T\|^{-1}\ep B_{C_0(M_\alpha)^*}
\]
and
\[
\cl{(T^*)^{-1} D}^{w^*} \cap A(M_\alpha) \text{ is a boundary of } \cl{(T^*)^{-1} D}^{w^*}.
\]
Using (\ref{eqn_C(K)_condition}), it follows that
\[
C \;\subseteq\; \cl{D}^{w^*} \;\subseteq\; C + \ep B_W \qquad\text{and}\qquad \cl{D}^{w^*} \cap A \text{ is a boundary of } \cl{D}^{w^*},
\]
which is what we need in order to satisfy Definition \ref{defn_admissible_set}.

All that remains is to construct the isomorphism $T$. Let $G=F \cap \omega\alpha$, $H=F\setminus \omega\alpha$ and choose a bijection
\[
\map{\pi}{G \cup (H \times \omega\alpha)}{\omega\alpha},
\]
such that $\pi(\lambda)=\lambda$ whenever $\lambda \in G$, and $q(\pi(\lambda,\eta))=q(\eta)$ whenever $(\lambda,\eta) \in H \times \omega\alpha$.

Given $\lambda \in H$, define the discrete union
\[
L_\lambda \;=\; \bigcup_{\eta < \omega\alpha} K_{\pi(\lambda,\eta)},
\]
which is a clopen subset of $M_\alpha$. Because $q(\pi(\lambda,\eta))=q(\eta)$ whenever $\eta < \omega\alpha$, it follows that the map
\[
(\eta)^\frown t \;\mapsto\; (\pi(\lambda,\eta))^\frown t, \qquad\qquad t \in M_{q(\eta)} \cup \{\varnothing\},
\]
is a homeomorphism of $K_\eta$ and $K_{\pi(\lambda,\eta)}$. Gluing together these homeomorphisms for  $\eta < \omega\alpha$ yields a homeomorphism of $M_\alpha$ and $L_\lambda$. Moreover, since $q(\lambda)=\alpha$, the map
\[
(\lambda)^\frown t \;\mapsto\; t, \qquad\qquad t \in M_\alpha,
\]
is a homeomorphism of $K_\lambda\setminus\{(\lambda)\}$ and $M_\alpha$. Together with Proposition \ref{prop_unif_bdd}, this means that there exists a surjective isomorphism
\[
\map{S_\lambda}{P_\lambda C_0(M_{\alpha+1}) \equiv C(K_\lambda)}{P_{L_\lambda}C_0(M_\alpha) \equiv C_0(L_\lambda)},
\]
such that $S^*_\lambda A(L_\lambda) = A(K_\lambda)$.

Finally, the properties of $\pi$ imply
\[
\bigcup_{\lambda \in G} K_\lambda \cup \bigcup_{\lambda \in H} L_\lambda \;=\; M_\alpha.
\]
Define $\map{T}{V}{C_0(M_\alpha)}$ by
\[
T \;=\; \sum_{\lambda \in G} P_\lambda + \sum_{\lambda \in H} S_\lambda P_\lambda.
\]
This has inverse
\[
T^{-1} \;=\; \sum_{\lambda \in G} P_\lambda + \sum_{\lambda \in H} S_\lambda^{-1} P_{L_\lambda}.
\]
Given $t \in M_\alpha$, we have $T^*\delta_t \in A(K_\lambda)$ if $t \in K_\lambda$ and $\lambda \in G$, and $T^*\delta_t \in A(K_\lambda)$ if $t \in L_\lambda$ and $\lambda \in H$. Thus $T^*A(M_\alpha) \subseteq A$. Arguing similarly using the inverse map yields equality, hence we have satisfied (\ref{eqn_C(K)_condition}), as required. The proof is complete.
\end{proof}

Thus we can apply Corollary \ref{cor_main_tool_2} to the spaces $C_0(M_\alpha)$ for all $\alpha \geqslant 1$, as promised. Our initial class, as advertised, consisted of compact spaces. Such a class can be easily arranged by considering $K_{\omega\alpha}$, which is homeomorphic to the 1-point compactification of $M_\alpha$ and thus has Cantor-Bendixson height $\alpha+1$.

We end this section with a problem. Note that if $K$ is a compact space such that $K^{(3)}$ is empty, then $C(K)$ admits an equivalent norm having a locally uniformly rotund dual norm \cite[Theorem VII.4.7]{dgz:93}. According to \cite[Theorem II.4.1]{dgz:93}, every equivalent norm on $C(K)$ can be approximated by such a norm, and all such norms are $C^1$-smooth \cite[Proposition I.1.5]{dgz:93}.

\begin{prob}\label{prob_K^2} Let $K$ be a compact space such that $K^{(3)}$ is empty? Can every equivalent norm on $C(K)$ be approximated by $C^2$-smooth norms or polyhedral norms?
\end{prob}

\end{document}